\documentclass[12pt]{amsart}
\usepackage{amsmath,amssymb,latexsym, amsthm, amscd, mathrsfs, stmaryrd, mathtools}
\usepackage[linktocpage=true]{hyperref}
\usepackage{cleveref}
\usepackage{color}
\usepackage[all]{xy}
\usepackage{enumitem}
\usepackage{chngcntr}

\crefformat{section}{\S#2#1#3} % see manual of cleveref, section 8.2.1
\crefformat{subsection}{\S#2#1#3}
\crefformat{subsubsection}{\S#2#1#3}

\newtheorem{thm}{Theorem} [section]
\theoremstyle{definition}

\theoremstyle{plain}
\newtheorem{prop}[thm]{Proposition}

\newtheorem{lem}[thm]{Lemma}
\newtheorem{cor}[thm]{Corollary}

\numberwithin{equation}{section}

\newcommand{\Bl}{\mathcal B}

\newcommand{\g}{\mathfrak{g}}
\newcommand{\bo}{\mathfrak{b}}
\newcommand{\gl}{\mathfrak{gl}}
\newcommand{\h}{\mathfrak{h}}

\newcommand{\n}{\mathfrak n}

\newcommand{\one}{{\ov 1}}

\newcommand{\ov}{\overline}
\newcommand{\OO}{\mathcal O}

\newcommand{\W}{\mathcal W}

\newcommand{\Z}{\mathbb Z}
\newcommand{\oo}{{\ov 0}}

\makeatletter
\newcommand{\extp}{\@ifnextchar^\@extp{\@extp^{\,}}}
\def\@extp^#1{\mathop{\bigwedge\nolimits^{\!#1}}}
\makeatother

\newcommand{\M}[1]{M_{#1}}

\title[Characters for Projective Modules in the BGG Category $\OO$]{Characters for Projective Modules in the BGG Category $\OO$ for general linear lie superalgebras }

\author[A.S. Kannan]{Arun S. Kannan}
\address{Department of Mathematics, University of Virginia, Charlottesville, VA 22904} \email{ask9ge@virginia.edu}

\begin{document}

\begin{abstract}
We determine the Verma multiplicities and the characters of projective modules for atypical blocks in the BGG Category $\OO$ for the general linear Lie superalgebras $\gl(2|2)$ and $\gl(3|1)$. We then explicitly determine the composition factor multiplcities of Verma modules in the atypicality $2$ block of $\gl(2|2)$.
\end{abstract}

\maketitle

\setcounter{tocdepth}{1}
\tableofcontents

%%%%%%%%%%
\section{Introduction}
% A Lie superalgebra is a generalization of a Lie algebra wherein we consider a $\mathbb{Z}_2$-graded vector space endowed with a Lie superbracket, which is subject to relations analagous to those a usual Lie bracket satisfies. The Lie superalgebra as a mathematical construct is particularly suited for theoretical discussions in particle physics regarding a phenomenon known as supersymmetry. Because a Lie superalgebra now has notions of an even and odd part, the super case introduces many interesting complexities to the representation theory of semisimple Lie algebras. For our inquiries, we will look at certain representations of the general linear Lie superalgebra $\gl(m|n)$, which informally can be thought of as all $(m+n)$-by-$(m+n)$ matrices. 
\subsection{}
The BGG Category $\OO$ is a category of modules of a semisimple Lie algebra that has been well studied for its rich and deep theory (cf. \cite{Hum08}). This category can be analogously defined for the general linear Lie superalgebra $\gl(m|n)$ (cf. \cite{CW12, Mus12}), and many results from the semisimple case extend to the super case. In this paper, we examine atypical blocks in the BGG Category $\OO$ for low-dimensional $\gl(m|n)$.
\subsection{}
Atypicality of weights is a phenomenon present in Lie superalgebras that has no analogue for semisimple Lie algebras. It allows for an integral block in $\OO$ whose degree of atypicality is greater than $0$ to have infinitely many simple modules. The principal block in $\OO$ for $\gl(m|n)$, which contains the trivial module, always has nonzero degree of atypicality when $m, n \geq 1$.
\par
Atypicality arises due to the presence of isotropic odd roots (i.e. roots of length zero) in the root system, which expand the notion of linkage beyond the orbit of the Weyl group. For $\gl(m|n)$, the degree of atypicality is an integer in the range $0$ to $\mathrm{min}(m, n)$, inclusive. The degree $0$ block can be reduced to the semisimple Lie algebra case via an equivalence of categories (cf. \cite{Gor02}). Therefore, the new cases arise primarily when the degree of atypicality is nonzero.
\subsection{}
Verma flag formulae for all tilting modules (and consequently, projective modules via BGG reciprocity and Soergel duality) in the Category $\OO$ of $\gl(m|n)$ are provided in \cite{CLW15}, proving the conjecture in \cite{Br03}. However, these formulae are given in terms of Brundan-Kazhdan-Lusztig polynomials, which are not easily computed and do not readily offer concrete multiplicities. Using these polynomials, the authors of \cite{CW08} were able to produce explicit standard filtration formulae for projectives in an atypical block of $\gl(2|1)$.

\subsection{} In this work, we use the tool of translation functors to determine the characters of projective modules in the BGG Category $\OO$ for the general linear Lie superalgebras $\gl(3|1)$, and $\gl(2|2)$ (after verifying the $\gl(2|1)$ case with the results in \cite{CW08}). Specifically, we explicitly determine the Verma multiplicities of standard filtration of projective modules in atypical blocks in $\OO$. For $\gl(3|1)$, we examine blocks of degree of atypicality $1$. There are infinitely many inequivalent atypical blocks. For $\gl(2|2)$, we examine blocks of degree of atypicality $2$. In this case, there is only one such integral block. Then, BGG reciprocity allows us to convert these formulae to formulae for composition multiplicities. We show this in the $\gl(2|2)$ case.
\par
\subsection{} Our general approach of using translation functors is as follows. Given some projective cover $P_\lambda$ for which we wish to deduce Verma multiplicities, we find some $P_\mu$ with known Verma multiplicities and some finite-dimensional representation $N$ such that the Verma module $M_\lambda$ appears in a standard filtration of $P_\mu \otimes N$. If $\lambda$ is the lowest weight appearing among all the weights linked to $\lambda$ appearing in the Verma flag, then $P_\lambda$ is a direct summand for the projection of $P_\mu \otimes N$ on to the block corresponding to $\lambda$. In most cases, it is the only direct summand. 
\par
A particularly useful set of criteria for determining whether a summand is direct and for verifying indecomposability is stated in Proposition \ref{filprop}. These criteria follow from similar criteria on tilting modules (cf. \cite{CW18}), which themselves are derived from the Super Jantzen sum formula (cf. \cite{Mus12}). Verifying indecomposability is a non-trivial step, as a priori it is not evident whether or not translation functors yield an indecomposable projective. See \cref{conditions} and \cref{strategy} for explicit details and justification.
\par
Our approach shows that in the cases we consider, standard filtrations always have Verma modules with multiplicity $1$ or $2$. By BGG reciprocity, these formulae determine the composition factors for Verma modules in $\OO$.
\subsection{}
In \cref{prelims}, we recall basic structure theorems for $\gl(m|n)$, fixing a Cartan subalgebra, a root system, a fundamental system, and defining linkage. Also, we recall the BGG Category $\OO$, review relevant results in the super case, and offer conditions when Verma modules appear in the standard filtration of projective modules.
\par
The sections \cref{sec5}, \cref{sec6}, and \cref{sec7} contain our original results. We find standard filtration multiplicities for projective modules of weights of degree of atypicality $1$ when $\g = \gl(3|1)$ and of degree of atypicality $2$ when $\g = \gl(2|2)$. These results are justified using the general facts in \cref{prelims} and the strategy of translation functors. We then compute the composition multiplicities of irreducibles for $\gl(2|2)$ for weights of degree of atypicality $2$.

\subsection*{\textbf{Acknowledgements}}This paper is a slightly modified version of the author's undergraduate thesis, which was supervised under Dr. Weiqiang Wang. Wang graciously extended the support of NSF grant DMS-1702254 to the author. This paper is based upon work supported by the National Science Foundation Graduate Research Fellowship Program under Grant No. 1842490.

\section{Preliminaries}\label{prelims}
We shall introduce some basic notation for the Lie superalgebra $\gl(m|n)$ in this section.
\subsection{Notation}
Suppose $V = \mathbb{C}^{m|n}$. Let $\{\ov{1}, \ov{2} \dots, \ov{m}\}$ and $\{1, 2, \dots, n\}$ parametrize the standard bases for the even and odd subspaces of $V$, $\mathbb{C}^m$ and $\mathbb{C}^n$, respectively. Denote
\begin{equation}
I(m,n) = \{\ov{1}, \ov{2}, \dots, \ov{m}; 1, 2, \dots, n\}
\end{equation}
where we impose the total order 

\begin{equation}
\ov{1} < \cdots < \ov{m} < 0 < 1 < \cdots < n.
\end{equation}
Here, $0$ is introduced for convenience. Let $\h$ be the Cartan subalgebra of diagonal matrices in $\gl(m|n)$, and let $\{\delta_i, \epsilon_j\}_{i,j}$ denote the basis of $\h^*$ dual to the canonical basis. Whenever convenient for $1 \leq i \leq m$, write 
\[\epsilon_{ \ov{i}} := \delta_i.\]
\par
The bilinear form $(\cdot, \cdot): \h \times \h \longrightarrow \mathbb{C}$ given by the supertrace $(x, y) \coloneqq \mathrm{str}(xy)$ naturally induces a bilinear form on $\h^*$, which will also be denoted by $(\cdot, \cdot)$. For $i,j\in I(m,n)$, we have:
\begin{equation}
(\epsilon_i, \epsilon_j) = 
\begin{cases}
 1 & \ov{1} \leq i = j \leq \ov{m} \\ 
-1 & 1 \leq i = j \leq n \\
 0 & i\neq j.
\end{cases}
\end{equation}
Note that 
\[(\delta_i - \epsilon_j, \delta_i - \epsilon_j) = 0\]
for $1 \leq i \leq m$ and $1 \leq j \leq n$. 
We can define the corresponding weight lattice $X$ in $\h^*$:
\begin{equation}
X \coloneqq \bigoplus_{i \in I(m,n)} \mathbb{Z}\epsilon_i.
\end{equation}
Let $\Phi$ denote the root system of $\gl(m|n)$ with respect to $\h^{*}$, and let $\Phi_{\ov{0}}$ and $\Phi_{\ov{1}}$ be the even and odd roots in $\Phi$, respectively. A root $\alpha \in \Phi$ is said to be {isotropic} if $(\alpha, \alpha) = 0$. Let $\bar{\Phi}_{\one}$ denote the set of all isotropic roots. Let $\Phi^+ = \{\epsilon_i - \epsilon_j\ | \ i < j \in I(m,n)\}$ and $\Pi = \{\epsilon_i - \epsilon_{i+1} \ | \ i \in I(m-1,n-1)\}\cup\{\epsilon_{\ov{m}} - \epsilon_1\}$ be a positive system and a fundamental system, respectively. Lastly, let $\mathcal{W} \cong S_n \times S_m$ denote the Weyl group of $\gl(m|n)$ with natural action on $\h^{*}$.
\par
Furthermore, we can define for any $\alpha \in \Phi_{\oo}$ the corresponding coroot $\alpha^\vee \in [\g_\alpha, \g_{-\alpha}] \subseteq \h$ such that 
\begin{equation}
\langle \lambda, \alpha^\vee \rangle = \frac{2(\lambda, \alpha)}{(\alpha, \alpha)}   \ \ \forall \lambda \in \h^*.
\end{equation}
The simple reflection $s_\alpha$ acts on $\h^*$ as expected: $s_\alpha(\lambda) = \lambda - \langle \lambda, \alpha^\vee \rangle \alpha$.
\par
Define the Weyl vector $\rho$ as follows:
\begin{equation}\label{normWeyl}
\rho = \sum_{i=1}^m (m-i+1)\delta_i - \sum_{j=1}^n j\epsilon_j.
\end{equation} 
\par
A weight $\lambda \in \h^*$ is said to be {antidominant} if $\langle \lambda + \rho, \alpha^\vee\rangle \not\in \mathbb{Z}_{> 0}$ and {dominant} if $\langle \lambda + \rho, \alpha^\vee\rangle \not\in \mathbb{Z}_{< 0}$ for all $\alpha \in \Phi^+_{\oo}$. 
\subsection{Atypicality and Linkage}\label{link}
The notion of linkage in the super case is similar to that of semisimple Lie algebras. However, the key distinction is that while blocks of modules in the semisimple Lie algebra case are finite, isotropic roots allow for blocks in the super case to be infinite. This arises because of a notion called atypicality.
\par
Let $\g = \gl(m|n)$ and $\h$ be the Cartan subalgebra of $\g$ consisting of the diagonal matrices with standard basis for $\h^*$ and standard choices for the root system as above.
\par
The {degree of atypicality} of $\lambda \in \h^*$, denoted $\#\lambda$, is the {maximum number of mutually orthogonal positive isotropic roots} $\alpha \in \bar{\Phi}_{\one}$ such that $(\lambda + \rho, \alpha) = 0$. An element $\lambda \in \h^*$ is said to be {typical} (relative to $\Phi^+$) if $\#\lambda = 0$ and is atypical otherwise.
\par
A relation $\sim$ on $\h^*$ can be defined as following. We say $\lambda \sim \mu \ \lambda,\mu \in \h^*$ if there exist mutually orthogonal isotropic odd roots $\alpha_1,\alpha_2,\dots,\alpha_l$, complex numbers $c_1,c_2,\dots,c_l$, and an element $w \in \W$ satisfying:
\begin{equation}
\mu + \rho = w\left(\lambda + \rho - \sum_{i=1}^{l}c_i\alpha_i\right), \ \ \ (\lambda + \rho, \alpha_i) = 0, i = 1\dots , l.
\end{equation}
The weights $\lambda$ and $\mu$ are said to be {linked} if $\lambda \sim \mu$. It can be shown that linkage is an equivalence relation.
\par
Given a fundamental root system $\Pi$, we can establish the {Bruhat order} on $\h^*$ as follows. Let $\lambda, \mu \in \h^*$. We say $\lambda \geq \mu$ if $\lambda \sim \mu$ and $\lambda - \mu \in \mathbb{Z}_{\geq 0}\Pi$ (i.e the nonnegative sum of simple roots). 
\par
Atypicality may not seem clear at first, so we introduce notation to elucidate the phenomenon. There exists a natural bijection between the integral weight lattice $X$ and $\mathbb{Z}^{m+n}$ where $\lambda \in X$ maps to $(q_1, q_2, \dots, q_m \ | \ r_1, r_2, \dots r_n) \in \mathbb{Z}^{m+n}$ if 
\[\lambda = \sum_{i=1}^{m}q_i\delta_i - \sum_{j=1}^{n}r_j\epsilon_j.\] 
Denote this identification with the congruence symbol $\cong$. By abuse of notation, we shall also use $\rho$ to denote the image $(m,m-1,\dots,1 \ | \ 1,2,\dots,n)$ of the  Weyl vector under this identification; the context will make it clear to which we refer. Furthermore, the action of the Weyl group $\W \cong S_m \times S_n$ is clear. We can permute everything to the left of the bar and to the right of the bar, but no coefficient may cross the bar.
\par
This bijection highlights atypicality very nicely. The degree of atypicality of the weight $(q_1, q_2, \dots, q_m \ | \ r_1, r_2, \dots r_n) - \rho$ is read by counting the number of pairs $(q_i, r_j)$ such that $q_i = r_j$, with the important stipulation no $q_i$ or $r_j$ be reused. The corresponding set of mutually orthogonal roots are $\delta_i - \epsilon_j$ for each pair $(i, j)$. The degree of the atypicality is also given by the size of the multiset $\{q_i\}_{i=1}^m\cap\{r_j\}_{j=1}^{n}$. In particular, if none of the $q_i$ coincide with the $r_j$, the weight is typical.
\subsection{Examples}\label{examples}
In this subsection, we take a concrete look at roots and the weight lattice.
\subsubsection{$\gl(3|1)$}\label{ex:gl31} With the Cartan subalgebra $\h$ given by the diagonal matrices, the bilinear form given by the supertrace induces a basis for $\h^*$ given by $\{\delta_1, \delta_2, \delta_3, \epsilon\}$, where we write $\epsilon$ to abbreviate $\epsilon_1$. By the standard convention above, the positive even roots are $\delta_1 - \delta_2$, $\delta_2 - \delta_3$, and $\delta_1 - \delta_3$, and the positive odd roots are $\delta_3 - \epsilon$, $\delta_2 - \epsilon$ and $\delta_1 - \epsilon$. The odd roots are also isotropic. The Weyl vector is given by $\rho = 3\delta_1 + 2\delta_2 +\delta_3 - \epsilon \cong (3,2,1 \ | 1)$.
\par
If $d \neq a,b,c$ are integers, then  $(a, b, c \ | \ d)  - \rho$  is typical, as there are no odd roots to which this weight is orthogonal. The weights $(a, b, c \ | \ c) - \rho$, $(a, b, c \ | \ b)  - \rho$, and $(a, b, c \ | \ a)  - \rho$ are atypical of degree $1$. After a $\rho$-shift, the first case is orthogonal to the odd root $\delta_3 - \epsilon$, the second is orthogonal to $\delta_2 - \epsilon$, and the last is orthogonal to $\delta_3 - \epsilon$. There are no other types of atypicality because none of the odd roots are pairwise orthogonal.
\par
The Weyl group is $\W \cong S_3 \times S_1$. We see that the integral atypical linkage classes are indexed by $a, b \in \mathbb{Z}, \ a \geq b$, with weights of the form $(a,b, c\ | \ c)  - \rho$, $(b,a, c\ | \ c)  - \rho$, $(a,c, b\ | \ c) - \rho$, $(b,c, a\ | \ c) - \rho$, $(c,a, b\ | \ c) - \rho$, and $(c,b, a\ | \ c) - \rho$ with $c \in \mathbb{Z}$ allowed to vary.
\par
Suppose $\lambda \cong (2,1, 3 \ | \ 1)$ and $\mu \cong (4,3,2 \ | \ 4)$. The weights $\lambda - \rho$ and $\mu  - \rho$ are linked because adding the odd root $\delta_2 - \epsilon$ (which is orthogonal to $\lambda$) to $\lambda$ thrice and then applying a Weyl group element yields $\mu$. 

\subsubsection{$\gl(2|2)$}\label{ex:gl22} With the Cartan subalgebra $\h$ given by the diagonal matrices, the bilinear form given by the supertrace induces a basis for $\h^*$ given by $\{\delta_1, \delta_2, \epsilon_1, \epsilon_2\}$. By the standard convention above, the positive even roots are $\delta_1 - \delta_2$ and $\epsilon_1 - \epsilon_2$. The positive odd roots (also isotropic) are $\delta_1 - \epsilon_1$, $\delta_1 - \epsilon_2$, $\delta_2 - \epsilon_1$, and $\delta_2 - \epsilon_2$. Observe now that we can choose two isotropic roots such that they are mutually orthogonal; one choice is $(\delta_1 - \epsilon_2, \delta_2 - \epsilon_1) = 0$ and the other is $(\delta_1 - \epsilon_1, \delta_2 - \epsilon_2) = 0$. This introduces weights of atypicality of degree $2$. The Weyl vector is given by $\rho = 2\delta_1 + \delta_1 - \epsilon_1 - 2\epsilon_2 \cong (2,1\ | \ 1, 2)$.
\par
If $a,b,c,d \in \mathbb{Z}$ and $\{a,b\} \cap \{c,d\} = \varnothing$, then $(a, b\ | \ c, d) - \rho$ is typical, as there are no odd roots to which this weight is is orthogonal after a $\rho$-shift. The difference from the previous two cases is that atypicality of degree $2$ is now possible. Because there are two pairs of two mutually orthogonal roots, weights of the form $(a, b \ | \ b, a) - \rho$ and $(a, b\ | \ a, b) - \rho$ are atypical of degree $2$.
\par
The Weyl group is $\W \cong S_2 \times S_2$. We see that there is one integral atypical linkage class of degree $2$, with weights of the form $(a, b \ | \ b, a) - \rho$, $(a, b \ | \ a, b) - \rho$, $(b, a \ | \ b, a) - \rho$, and $(b, a \ | \ a, b) - \rho$, where $a \geq b \in \mathbb{Z}$ are free to vary. For example, the weight $(2, 1 \ | \ 1, 2) -\rho$ is linked to $(3, 5 \ | \ 5, 3) - \rho$, but not to $(3, 3 \ | \ 5, 5) - \rho$.
\par
Suppose $\lambda \cong (2,1 \ | \ 1, 2)$ and $\mu \cong (5, 8 \ | \ 5 \ 8)$. The weights $\lambda - \rho$ and $\mu - \rho$ are linked because adding the odd root $\delta_2 - \epsilon_1$ four times and $\delta_1 - \epsilon_2$ six times to $\lambda$ and then applying a Weyl group element yields $\mu$. Observe that these odd roots are both orthogonal to $\lambda$ and are orthogonal to each other.
\subsection{The BGG Category $\OO$}
From now on, let $\g = \gl(m|n) = \g_{\oo} \oplus \g_{\one}$ with the standard associated bilinear form, root system, and triangular decomposition: $\g = \n^-\oplus\h\oplus\n^+$ and $\bo = \h\oplus\n^+$. Recall that the {BGG category} $\OO$ is the full subcategory of $U(\g)$-modules $M$ subject to the following three conditions:
\begin{enumerate}
\item $M$ is finitely generated.
\item $M$ is $\h$-semisimple: $M = \bigoplus_{\lambda \in X} M^\lambda$, where $M^\lambda = \{v \in M \ | \ h \cdot v = \lambda(h)v \ \ \forall h \in \h \}$ is a nonzero weight space.
\item $M$ is locally $\n^+$-finite: $U(\n^+)\cdot v$ is finite dimensional for $\forall v \in M$. 
\end{enumerate}
Observe that the abelian quotient algebra $\bo /\n^+ \cong \h$. Thus, any $\lambda \in \h^*$ naturally defines a one-dimensional $\bo$-module with trivial $\n^+$-action, which we denote as $\mathbb{C}_\lambda$. Specifically, if $v \in \mathbb{C}_\lambda$, then $h\cdot v = \lambda(h) v, \ \forall h \in \h$. Now, define
\begin{equation}
M_\lambda \coloneqq U(\g)\otimes_{U(\bo)}\mathbb{C}_{\lambda - \rho},
\end{equation}
where $\rho$ is the Weyl vector. This is naturally a left $\g$-module. This is called a {Verma module} of highest weight $\lambda - \rho$.
\par
We let $L_\lambda$ denote the unique simple quotient of $M_\lambda$ of highest weight $\lambda - \rho$, and use the notation $[M_\mu : L_\lambda]$ to denote the multiplicity of $L_\lambda$ in a composition series of $M_\mu$. Such a series exists for all $M$ in $\OO$.
\par
In the notation introduced in \cref{link}, if $\lambda \cong (q_1\dots q_m \ | \ r_1\dots r_n)$, write $M_{q_1\dots q_m | r_1\dots r_n}$ to denote $M_{\lambda}$ and $L_{q_1\dots q_m | r_1\dots r_n}$ to denote $L_{\lambda}$. 
\subsection{Blocks in $\OO$}\label{blocks}
The integral blocks in $\OO$ can be divided into typical and atypical blocks. By definition, any simple module in a typical block has typical highest weight. By Gorelik \cite{Gor02}, any integral typical block in $\OO$ is equivalent to a block in the BGG Category of $\g_{\oo}$-modules. It remains to better understand the atypical blocks.
\par
Now, recall the examples in \cref{examples}. In $\gl(3|1)$, the linked weights are $(a,b, c\ | \ c)  - \rho$, $(b,a, c\ | \ c)  - \rho$, $(a,c, b\ | \ c) - \rho$, $(b,c, a\ | \ c) - \rho$, $(c,a, b\ | \ c) - \rho$, and $(c,b, a\ | \ c) - \rho$ with $a,b,c \in \mathbb{Z}$ and $c$ allowed to vary. We will let $\Bl_{a,b}$ denote the corresponding block. Lastly, in $\gl(2|2)$, there is only one block of atypicality degree $2$; we will denote it as $\Bl_0$.

\subsection{Key Results in $\OO$}
The primary means by which the goals of this paper are achieved are by using translation functors. We restate the necessary results to justify our steps. This collection of results is justified in \cite[Chap. 1-3]{Hum08} for the BGG Category $\OO$ for semisimple Lie algebras; similar arguments extend them to the BGG Category $\OO$ of $\gl(m|n)$-modules.
\begin{thm}\label{tran}
Let $\g = \gl(m|n)$. Let $N$ be a finite dimensional $U(\g)$-module. For any $\lambda \in \h^*$, the tensor module $T \coloneqq M_\lambda \otimes N$ has a finite filtration with quotients isomorphic to Verma modules of the form $M_{\lambda+\mu}$, where $\mu$ ranges over the weights of $N$, each occurring $\mathrm{dim} \ N^\mu$ times in the filtration.
\end{thm}
\par
A module $N \in \OO$ has a {standard filtration} or a {Verma flag} if there is a sequence of submodules $0 = N_0 \subset N_1 \subset N_2 \subset \cdots N_k = N$ such that each $N_i/N_{i-1} \ 1\leq i \leq k$ is isomorphic to a Verma module.  The number of times the Verma module $M_\lambda$ appears in a standard filtration of $N$ is denoted by $(N : M_\lambda)$. 
\par
It can be shown that the length and the Verma multiplicities in a standard filtration are independent of choice of a standard filtration. Therefore,  the following informal notation to indicate a standard filtration of a module is useful. If $M_{\lambda_i}, \ \lambda_i \in \h^*, \ 1 \leq i \leq k$ are the Verma modules appearing with multiplicity $c_i \in \mathbb{Z}_{>0}$ in a standard filtration of a module $N$, we write:
\begin{equation}
N = c_1M_{\lambda_1} + c_2M_{\lambda_2} + \dots + c_kM_{\lambda_k}
\end{equation}
\par
Similarly, if $L_{\mu_i}, \ \mu_i \in \h^{*}, \ 1 \leq i \leq k$ are the irreducibles appearing with multiplicity $d_i \in \Z_{> 0}$ in a composition series of a module $N$, we write
\begin{equation}
N = d_1L_{\mu_1} + d_2L_{\mu_2} + \dots + d_kL_{\mu_k}
\end{equation}
\par
We let $P_\lambda$ denote the (unique) projective cover for $L_\lambda$ for all $\lambda \in \h^*$, that is the indecomposable projective such that $P_\lambda \twoheadrightarrow L_\lambda \rightarrow 0$. We recall the following facts about projectives.
\begin{enumerate}
\item All projectives have a standard filtration. \label{p0}
\item The Category $\OO$ has enough projectives. \label{p1}
\item The Verma modules $M_\mu$ which appear in a standard filtration of $P_\lambda$ satisfy $\mu \geq \lambda$ in the Bruhat ordering, and $M_\lambda$ appears with multiplicity $1$. \label{p4}
\end{enumerate}
The following proposition, which follows from Theorem \ref{tran}, is a critical part of our translation functor arguments.
\begin{prop}\label{sum}
If a projective $P$ has a standard filtration given by $P_\lambda = \sum_{\nu} M_{\nu}$, the $\nu$ not necessarily distinct, then for any finite-dimensional representation $N$ with weights $\mu$, the standard filtration for $P\otimes N$ is given by $\sum_{\nu}\sum_{\mu} M_{\nu + \mu}$, where $\mu$ appears in the sum with multiplicity given by $\mathrm{dim}\ N^\mu$.
\end{prop}
 The following lemma will also be useful in our arguments. It is predicated on the fact that the Weyl groups for $\gl(3|1)$ and $\gl(2|2)$ are products of dihedral groups (cf. \cite{CW18}).
\begin{lem}\label{typAndDom}
If $\g = \gl(m|n)$ and $m, n \leq 3$ and $\lambda \in X$ is typical, then the Verma modules that appear in a standard filtration of $P_\lambda$ are of the form $M_{w\lambda}$, where $w \in \W$ such that $w\lambda \geq \lambda$, and each Verma module appears with multiplicity $1$. 
\end{lem}
Lastly, we recall BGG reciprocity.
\begin{equation}\label{BGGrecip}
(P_\lambda : M _\mu) = [M_\mu : L_\lambda], \ \ \lambda, \mu \in \h^* .
\end{equation}
\subsection{Some Representations of $\gl(m|n)$}
The strategy of using translation functors involves choosing appropriate representations to tensor with projective modules to produce new modules.
\par
The {natural representation} $V = \mathbb{C}^{m|n}$ be of $\g = \gl(m|n)$ has weights $\{ \epsilon_i \ | \ i \in I(m,n)\}$, while the {dual representation} $V^*$ has weights $\{ -\epsilon_i \ | \ i \in I(m,n)\}$. 
\par
The {exterior algebra} of a finite-dimensional vector superspace. Let $W = W_{\oo}\oplus W_{\one}$ be a vector superspace. Then, we can define the $k$-th exterior power of $W$ as follows:
\begin{equation}
\extp^k(W) \coloneqq \bigoplus_{i+j=k} \left( \Lambda^i(W_\oo)\otimes \mathrm{S}^j(W_{\one}) \right)
\end{equation} 
where $\Lambda^i$ and $\mathrm{S}^j$ acting on vector spaces are the $i$-th exterior power and $j$-th symmetric power in the traditional sense, respectively. We will particularly be interested in the $k$-th exterior power when $k = 2$ or $k = 3$ and $W = V$ or $W = V^*$, which we refer to as wedge-squared or wedge-cubed of the natural or of the dual, respectively.

\subsection{Conditions for nonzero Verma flag multiplicities in projective modules}\label{conditions}
We have the following proposition, which uses BGG reciprocity to reformulate the conditions for tilting modules in \cite[Proposition 2.2]{CW18} as conditions for projective modules.
\begin{prop}\label{filprop}
 
Suppose that $\lambda \in X, \alpha_i \in \Phi_{\bar{0}}^+, 1 \leq i \leq k,$ and $\beta, \gamma \in \Phi_{\bar{1}}^+$. Let $w = \prod_{i=1}^k s_{\alpha_i} \in \W$. 
	\begin{enumerate}
	    \item Suppose that $\langle\lambda, \alpha_1^\vee\rangle < 0$. Then $(P_{\lambda} : M_{s_{\alpha_1}\lambda}) > 0$. \label{1}
	    
	    \item Suppose that $\langle s_{\alpha_{i-1}}\cdots s_{\alpha_1}\lambda, \alpha_i^\vee\rangle < 0 \ \forall i \in {1,2,\dots,k}$. then $(P_{\lambda} : M_{w\lambda}) > 0$. \label{2}
	    
	    \item  Suppose that $(\lambda, \beta) = 0$. Then $(P_{\lambda} : M_{\lambda + \beta}) > 0$. \label{3}
	    
	    \item Suppose that $(\lambda, \beta) = 0$ and  $\langle s_{\alpha_{i-1}}\cdots s_{\alpha_1}(\lambda + \beta), \alpha_i^\vee\rangle < 0 \ \forall i \in {1,2,\dots,k}$. Then $(P_{\lambda} : M_{w(\lambda + \beta)}) > 0$. \label{4}
	    
	    \item Suppose that $(\lambda, \beta) = (\lambda + \beta, \gamma) = 0$ and $\mathrm{ht}(\beta) < \mathrm{ht}(\gamma)$. Then $(P_{\lambda} : M_{\lambda + \beta + \gamma}) > 0$. \label{5}
	    
	    \item Suppose that $(\lambda, \beta) = (\lambda + \beta, \gamma) = 0$, $\mathrm{ht}(\beta) < \mathrm{ht}(\gamma)$, and $\langle s_{\alpha_{i-1}}\cdots s_{\alpha_1}(\lambda + \beta + \gamma), \alpha_i^\vee\rangle < 0 \ \forall i \in {1,2,\dots,k}$. Then $(P_{\lambda} : M_{w(\lambda + \beta + \gamma)}) > 0$. \label{6}
\end{enumerate}
\end{prop}

\begin{proof}
The proposition is originally derived using the Super Jantzen sum formula (cf. \cite{Gor02, Mus12}), giving conditions for composition factors. BGG Reciprocity \eqref{BGGrecip} immediately translates the conditions from those on tilting modules to those on projective modules.
\end{proof}

\begin{cor}\label{corlen}
Suppose $\lambda - \rho \in \h^*$ is atypical. Then $P_\lambda$ must have a Verma flag of length greater than $1$. 
\end{cor}

\begin{proof}
$M_\lambda$ appears in the standard filtration. Furthermore, because $\lambda$ is atypical, there exists $\beta$ such that $\beta \in \Phi_{\one}^+$ and $(\lambda, \beta) = 0$. Therefore, apply Proposition \ref{filprop}[\ref{4}] to see that $M_{\lambda + \beta}$ also appears in the standard filtration.
\end{proof}

\subsection{Strategy}\label{strategy}
Given an atypical $\lambda -\rho \in \h^*$, we seek to deduce the standard filtration formula of $P_{\lambda}$. To do so, we choose a $\mu \in \h^*$ such that we know a standard filtration for $P_{\mu}$. This is often accomplished by letting $\mu \coloneqq \lambda - \nu$, where $\nu$ is the lowest weight in some finite-dimensional representation $W$ such that $\mu - \rho$ is typical; Lemma \ref{typAndDom} tells us the structure of $P_{\mu}$. Proposition \ref{sum} can be used to deduce the Verma modules which appear in a standard filtration of the projective $P_{\mu} \otimes W$, which must include $M_{\lambda}$. Our next step is to project $P_{\mu} \otimes W$ onto the block corresponding to the linkage class of $\lambda - \rho$. We denote the resulting projection as $\mathrm{Pr}_\lambda(P_\mu \otimes W)$. If $M_{\lambda}$ has the lowest weight of all the Verma modules in the standard filtration of the projection, $P_\lambda$ must appear in that projection as a direct summand, as the direct summands of a projective are projective. The projection itself is done by collecting all Verma modules in the standard filtration whose weights are linked to $\lambda - \rho$.
\par
In this projection, we apply Proposition \ref{filprop} to see which Verma modules appear in the standard filtration of $P_\lambda$. These necessarily appear in the projection because $P_\lambda$ is a direct summand. Then, we generally try to argue that there is no other direct summand (i.e. $P_\lambda$ is the projection). This is often done with the help of Corollary \ref{corlen}. 
\par
As a remark, it is not always necessary to take $\mu \coloneqq \lambda - \nu$, where $\nu$ is the lowest weight in the representation $W$. This is often a good initial choice, but the key requirement is that $\lambda$ be the lowest weight appearing in the standard filtration after the projection on to the block corresponding to the linkage class of $\lambda - \rho$.

\section{Character Formulae for $\gl(3|1)$}\label{sec5}
In this section, we determine Verma multiplicities for standard filtration formulae for projective covers of simple modules of $\gl(3|1)$ with integral, atypical weight of degree $1$.
\subsection{Results}
Let $\g = \gl(3|1)$ have the standard choices of Cartan subalgebra, bilinear form, root system, positive, and fundamental system as described in \cref{prelims}.  Recall the notation described in \cref{link} to describe a weight in $\h^*$. Lastly, recall Example \ref{ex:gl31} and the corresponding blocks $\Bl_{a,b}$, $a,b\in \mathbb{Z}$ (see \cref{blocks}). We have the following Theorems \ref{gl311} to \ref{gl316} that describe standard filtrations of projectives in these blocks.
\begin{thm}\label{gl311}
Let $a, b, c \in \mathbb{Z}$ with $a \geq b$. The projective objects $P_{a, b, c| c}$ in $\Bl_{a,b}$ have the following Verma flag formulae.
\begin{enumerate}[label=(\arabic*), ref=(\arabic*]
	\item Suppose $b > c$.
	\begin{enumerate}[label=\theenumi.\arabic*), ref = \arabic*]
		\item If $b > c+1$, then
        \begin{equation*}
        \begin{aligned}            
        P_{a, b, c| c} = M_{a,b,c|c} + M_{a,b,c+1|c+1}.
        \end{aligned}
        \end{equation*}		
        
		\item Suppose $b = c+1$.
		\begin{enumerate}[label=\theenumi.\theenumii.\arabic*)]
			\item If $a > c+2$, then
	        \begin{equation*}
	        \begin{aligned}
	        P_{a,c+1,c|c} = M_{a,c+1,c|c} + M_{a,c+1,c+1|c+1} + M_{a,c+2,c+1|c+2}.            
	        \end{aligned}
	        \end{equation*}				
			\item If $a = c+2$, then
	        \begin{equation*}
	        \begin{aligned}            
	        P_{c+2,c+1,c|c} = M_{c+2,c+1,c|c} + M_{c+2,c+2,c+1|c+2} \\
	        + M_{c+2,c+1,c+1|c+1} + M_{c+3,c+2,c+1|c+3}.
	        \end{aligned}
	        \end{equation*}				
			\item If $a = c+1$, then
	        \begin{equation*}
	        \begin{aligned}            
	        P_{c+1,c+1,c|c} = M_{c+1,c+1,c|c} + M_{c+1,c+1,c+1|c+1} \\ 
	        + M_{c+1,c+2,c+1|c+2} + M_{c+2,c+1,c+1|c+2}.
	        \end{aligned}
	        \end{equation*}				
		\end{enumerate}
	\end{enumerate}
	\item Suppose $b = c$.
	\begin{enumerate}[label=\theenumi.\arabic*), ref = \arabic*]
		\item If $a > c+1$, then
	        \begin{equation*}
	        \begin{aligned}            
	        P_{a,c,c|c} =M_{a,c,c|c} + M_{a,c,c+1|c+1} + M_{a,c+1,c|c+1}.
	        \end{aligned}	
	        \end{equation*}	        
		\item If $a = c+1$, then
	        \begin{equation*}
	        \begin{aligned}            
	        P_{c+1,c,c|c} = M_{c+1,c,c|c} + M_{c+2,c+1,c|c+2} + M_{c+1,c,c+1|c+1}\\
	         + M_{c+1,c+1,c|c+1} + M_{c+2,c,c+1|c+2}.
	        \end{aligned}
	        \end{equation*}		        
		\item If $a = c$, then
	        \begin{equation*}
	        \begin{aligned}            
	        P_{c,c,c|c} = M_{c,c,c|c} + M_{c,c,c+1|c+1} \\ 
	        + M_{c+1,c,c|c+1} + M_{c+1,c,c|c+1}.
	        \end{aligned}
	        \end{equation*}			
	\end{enumerate}	
	\item Suppose $b < c$.
	\begin{enumerate}[label=\theenumi.\arabic*), ref = \arabic*]
		\item Suppose $a > c$.
		\begin{enumerate}[label=\theenumi.\theenumii.\arabic*)]
			\item If $a > c+1$, then
	        \begin{equation*}
	        \begin{aligned}            
	        P_{a,b,c|c} = M_{a,b,c|c} + M_{a,b,c+1|c+1}\\ + M_{a,c,b|c} + M_{a,c+1,b|c+1}.
	        \end{aligned}
	        \end{equation*}							
			\item If $a = c+1$, then
	        \begin{equation*}
	        \begin{aligned}            
	        P_{c+1,b,c|c} = M_{c+1,b,c|c} + M_{c+1,b,c+1|c+1} + M_{c+2,b,c+1|c+2} \\ 
	        + M_{c+1,c,b|c} + M_{c+1,c+1,b|c+1} + M_{c+2,c+1,b|c+2}.
	        \end{aligned}
	        \end{equation*}				
		\end{enumerate}
		\item If $a = c$, then	
	        \begin{equation*}
	        \begin{aligned}            
	        P_{c,b,c|c} =M_{c,b,c|c} + M_{c,b,c+1|c+1} + M_{c+1,b,c|c+1} \\ 
	        + M_{c,c,b|c} + M_{c,c+1,b|c+1} + M_{c+1,c,b|c+1}.
	        \end{aligned}
	        \end{equation*}			
		\item Suppose $a < c$.
		\begin{enumerate}[label=\theenumi.\theenumii.\arabic*)]
			\item If $a > b$, then
	        \begin{equation*}
	        \begin{aligned}            
	        P_{a,b,c|c} = M_{a,b,c|c} + M_{a,b,c+1|c+1} + M_{a,c,b|c}\\  + M_{c,a,b|c} + M_{a,c+1,b|c+1} + M_{c+1,a,b|c+1}\\ + M_{c,b,a|c} + M_{c+1,b,a|c+1}.
	        \end{aligned}
	        \end{equation*}			
			\item If $a = b$, then 
	        \begin{equation*}
	        \begin{aligned}            
	        P_{b,b,c|c} =M_{b,b,c|c} + M_{b,b,c+1|c+1} + M_{b,c,b|c}\\ 
	         + M_{b,c+1,b|c+1} + M_{c,b,b|c} + M_{c+1,b,b|c+1}.
	        \end{aligned}
	        \end{equation*}				
			
		\end{enumerate}		
	\end{enumerate}		
\end{enumerate}
\end{thm}

\begin{thm}\label{gl312}
Let $a, b, c \in \mathbb{Z}$ with $a \geq b$. The projective objects $P_{b, a, c| c}$ in $\Bl_{a,b}$ have the following Verma flag formulae.
\begin{enumerate}[label=(\arabic*), ref=(\arabic*]
	\item Suppose $b > c$.
	\begin{enumerate}[label=\theenumi.\arabic*), ref = \arabic*]
		\item If $b > c+1$ and $a > b$, then
        \begin{equation*}
        \begin{aligned}            
        P_{b, a, c| c} = M_{b,a,c|c} + M_{b,a,c+1|c+1} + M_{a,b,c|c} + M_{a,b,c+1|c+1}.
        \end{aligned}
        \end{equation*}		
        
		\item Suppose $b = c+1$.
		\begin{enumerate}[label=\theenumi.\theenumii.\arabic*)]
			\item If $a > c+2$, then
	        \begin{equation*}
	        \begin{aligned}
	        P_{c+1,a,c|c} = M_{c+1,a,c|c} + M_{c+2,a,c+1|c+2} + M_{c+1,a,c+1|c+1}\\ 
	        + M_{a,c+1,c|c} + M_{a,c+2,c+1|c+2}+ M_{a,c+1,c+1|c+1}.            
	        \end{aligned}
	        \end{equation*}				
			\item If $a = c+2$, then
	        \begin{equation*}
	        \begin{aligned}            
	        P_{c+1 ,c+2,c|c} = M_{c+1,c+2,c|c} + M_{c+2,c+1,c|c}  + M_{c+1,c+2,c+1|c+1}  \\
	        + M_{c+2,c+1,c+1|c+1} + M_{c+2,c+2,c+1|c+2}.
	        \end{aligned}
	        \end{equation*}						
		\end{enumerate}
	\end{enumerate}
	\item Suppose $b = c$.
	\begin{enumerate}[label=\theenumi.\arabic*), ref = \arabic*]
		\item If $a > c+1$, then
	        \begin{equation*}
	        \begin{aligned}            
	        P_{c,a,c|c} =M_{c,a,c|c} + M_{c,a,c+1|c+1} + M_{c+1,a,c|c+1} \\ + M_{a,c,c|c} + M_{a,c,c+1|c+1} + M_{a,c+1,c|c+1}.
	        \end{aligned}	
	        \end{equation*}	        
		\item If $a = c+1$, then
	        \begin{equation*}
	        \begin{aligned}            
	        P_{c,c+1,c|c} = M_{c,c+1,c|c} + M_{c,c+2,c+1|c+2} + M_{c+1,c+2,c|c+2}  + M_{c,c+1,c+1|c+1}\\ + 2M_{c+1,c+1,c|c+1} + M_{c+2,c,c+1|c+2} + M_{c+2,c+1,c|c+2} \\ + M_{c+1,c,c+1|c+1} + M_{c+1,c,c|c}.
	        \end{aligned}
	        \end{equation*}		        	
	\end{enumerate}	
	\item Suppose $b < c$.
	\begin{enumerate}[label=\theenumi.\arabic*), ref = \arabic*]
		\item Suppose $a > c$.
		\begin{enumerate}[label=\theenumi.\theenumii.\arabic*)]
			\item If $a > c+1$, then
	        \begin{equation*}
	        \begin{aligned}            
	        P_{b,a,c|c} = M_{b,a,c|c} + M_{a,b,c|c} + M_{a,c,b|c} + M_{c,a,b|c} + M_{b,a,c+1|c+1}\\ + M_{a,b,c+1|c+1} + M_{a,c+1,b|c+1} + M_{c+1,a,b|c+1}.
	        \end{aligned}
	        \end{equation*}							
			\item If $a = c+1$, then
	        \begin{equation*}
	        \begin{aligned}            
	        P_{b,c+1,c|c} = M_{b,c+1,c|c} + M_{b,c+2,c+1|c+2} + M_{b,c+1,c+1|c+1} \\
            + M_{c+1,b,c|c} + M_{c+2,b,c+1|c+2} + M_{c+1,b,c+1|c+1} \\
            + M_{c,c+1,b|c} + M_{c+1,c+2,b|c+2} + M_{c+1,c+1,b|c+1} \\
            + M_{c+1,c,b|c} + M_{c+2,c+1,b|c+2} + M_{c+1,c+1,b|c+1}.
	        \end{aligned}
	        \end{equation*}				
		\end{enumerate}
		\item If $a = c$, then	
	        \begin{equation*}
	        \begin{aligned}            
	        P_{b,c,c|c} =  M_{b,c,c|c} + M_{b,c,c+1|c+1} + M_{b,c+1,c|c+1} \\
            + M_{c,b,c|c} + M_{c,b,c+1|c+1} + M_{c+1,b,c|c+1} \\
            + M_{c,c,b|c} + M_{c,c+1,b|c+1} + M_{c+1,c,b|c+1}.
	        \end{aligned}
	        \end{equation*}			
		\item Suppose $a < c$ and $a > b$.
	        \begin{equation*}
	        \begin{aligned}            
	        P_{b,a,c|c} = M_{b,a,c|c} + M_{b,a,c+1|c+1} + M_{b,c,a|c} \\
            + M_{b,c+1,a|c+1} + M_{a,b,c|c} + M_{a,b,c+1|c+1} \\
            + M_{c,b,a|c} + M_{c+1,b,a|c+1} + M_{a,c,b|c} \\
            + M_{a,c+1,b|c+1} + M_{c,a,b|c} + M_{c+1,a,b|c+1}.
	        \end{aligned}
	        \end{equation*}		
	\end{enumerate}		
\end{enumerate}
\end{thm}

\begin{thm}\label{gl313}
Let $a, b, c \in \mathbb{Z}$ with $a \geq b$. The projective objects $P_{a, c, b| c}$ in $\Bl_{a,b}$ have the following Verma flag formulae.
\begin{enumerate}[label=(\arabic*), ref=(\arabic*]
	\item Suppose $b > c$.
	\begin{enumerate}[label=\theenumi.\arabic*), ref = \arabic*]
		\item If $b > c+1$, then
        \begin{equation*}
        \begin{aligned}            
        P_{a, c, b| c} = M_{a,c,b|c} + M_{a,c+1,b|c+1} + M_{a,b,c|c} + M_{a,b,c+1|c+1}.
        \end{aligned}
        \end{equation*}		
        
		\item Suppose $b = c+1$.
		\begin{enumerate}[label=\theenumi.\theenumii.\arabic*)]
			\item If $a > c+1$, then
	        \begin{equation*}
	        \begin{aligned}
	        P_{a,c,c+1|c} = M_{a,c,c+1|c} + M_{a,c+1,c+1|c+1} + M_{a,c+1,c|c}.           
	        \end{aligned}
	        \end{equation*}			
			\item If $a = c+1$, then
	        \begin{equation*}
	        \begin{aligned}            
	        P_{c+1 ,c,c+1|c} = M_{c+1,c,c+1|c} + M_{c+1,c+1,c+1|c+1} + \\M_{c+2,c+1,c+1|c+2} + M_{c+1,c+1,c|c}.
	        \end{aligned}
	        \end{equation*}		        						
		\end{enumerate}
	\end{enumerate}
	\item Suppose $b < c$.
	\begin{enumerate}[label=\theenumi.\arabic*), ref = \arabic*]
		\item Suppose $a > c$.
		\begin{enumerate}[label=\theenumi.\theenumii.\arabic*)]
			\item If $a > c+1$, then
	        \begin{equation*}
	        \begin{aligned}            
	        P_{a,c,b|c} = M_{a,c,b|c} + M_{a,c+1,b|c+1}.
	        \end{aligned}
	        \end{equation*}							
			\item If $a = c+1$, then
	        \begin{equation*}
	        \begin{aligned}            
	        P_{c+1,c,b|c} = M_{c+1,c,b|c} + M_{c+1,c+1,b|c+1} + M_{c+2,c+1,b|c+2}.
	        \end{aligned}
	        \end{equation*}				
		\end{enumerate}
		\item If $a = c$, then	\label{help0}
	        \begin{equation*}
	        \begin{aligned}            
	        P_{c,c,b|c} =  M_{c,c,b|c} + M_{c,c+1,b|c+1} + M_{c+1,c,b|c+1}.
	        \end{aligned}
	        \end{equation*}			
		\item If $a < c$, then
	        \begin{equation*}
	        \begin{aligned}            
	        P_{a,c,b|c} = M_{a,c,b|c} + M_{a,c+1,b|c+1} + M_{c,a,b|c} + M_{c+1,a,b|c+1}.
	        \end{aligned}
	        \end{equation*}		
	\end{enumerate}		
\end{enumerate}
\end{thm}

\begin{thm}\label{gl314}
Let $a, b, c \in \mathbb{Z}$ with $a \geq b$. The projective objects $P_{b, c, a| c}$ in $\Bl_{a,b}$ have the following Verma flag formulae.
\begin{enumerate}[label=(\arabic*), ref=(\arabic*]
	\item Suppose $b > c$.
	\begin{enumerate}[label=\theenumi.\arabic*), ref = \arabic*]
		\item If $b > c+1$ and $a > b$, then
        \begin{equation*}
        \begin{aligned}            
            P_{b, c, a| c} =   M_{b,c,a|c} + M_{b,a,c|c} + M_{a,b,c|c} + M_{a,c,b|c}
            + M_{b,c+1,a|c+1}\\ + M_{b,a,c+1|c+1} + M_{a,b,c+1|c+1} + M_{a,c+1,b|c+1}.
        \end{aligned}
        \end{equation*}		
        
		\item Suppose $b = c+1$.
		\begin{enumerate}[label=\theenumi.\theenumii.\arabic*)]
			\item If $a > c+2$, then
	        \begin{equation*}
	        \begin{aligned}
	        P_{c+1,c,a|c} =    M_{c+1,c,a|c} + M_{c+2,c+1,a|c+2} + M_{c+1,c+1,a|c+1} \\
            + M_{a,c,c+1|c} + M_{a,c+1,c+2|c+2} + M_{a,c+1,c+1|c+1} \\
            + M_{c+1,a,c|c} + M_{c+1,a,c+2|c+2} + M_{c+1,a,c+1|c+1} \\
            + M_{a,c+1,c|c} + M_{a,c+2,c+1|c+2} + M_{a,c+1,c+1|c+1}. \\        
	        \end{aligned}
	        \end{equation*}				
			\item If $a = c+2$, then
	        \begin{equation*}
	        \begin{aligned}            
	        P_{c+1 ,c,c+2|c} =  M_{c+1,c,c+2|c} + M_{c+2,c,c+1|c} + M_{c+1,c+1,c+2|c+1}\\ + M_{c+2,c+1,c+1|c+1}
            + M_{c+2,c+1,c+2|c+2} + M_{c+2,c+1,c|c}\\ + M_{c+1,c+2,c|c} + M_{c+1,c+2,c+1|c+1}
            + M_{c+2,c+1,c+1|c+1}\\ + M_{c+2,c+2,c+1|c+2}.
	        \end{aligned}
	        \end{equation*}						
		\end{enumerate}
	\end{enumerate}
	\item Suppose $b = c$.
	\begin{enumerate}[label=\theenumi.\arabic*), ref = \arabic*]
		\item If $a > c+1$, then
	        \begin{equation*}
	        \begin{aligned}            
	        P_{c,c,a|c} = M_{c,c,a|c} + M_{c+1,c,a|c+1} + M_{c,c+1,a|c+1} \\
            + M_{c,a,c|c} + M_{c+1,a,c|c+1} + M_{c,a,c+1|c+1} \\
            + M_{a,c,c|c} + M_{a,c+1,c|c+1} + M_{a,c,c+1|c+1}.
	        \end{aligned}	
	        \end{equation*}	        
		\item If $a = c+1$, then
	        \begin{equation*}
	        \begin{aligned}            
	        P_{c,c,c+1|c} =  M_{c,c,c+1|c} + M_{c+1,c,c|c} + M_{c,c+1,c|c} \\
            + M_{c,c+1,c+1|c+1} + M_{c+1,c,c+1|c+1} + M_{c+1,c+1,c|c+1}.          
	        \end{aligned}
	        \end{equation*}		        	
	\end{enumerate}		
	\item Suppose $b < c$.
	\begin{enumerate}[label=\theenumi.\arabic*), ref = \arabic*]
			\item If $a > c+1$, then
	        \begin{equation*}
	        \begin{aligned}            
	        P_{b,c,a|c} = M_{b,c,a|c} + M_{b,c+1,a|c+1} + M_{c,b,a|c} + M_{c+1,b,a|c+1}\\ + M_{a,b,c|c} + M_{a,b,c+1|c+1}
	        			+ M_{a,c,b|c} + M_{a,c+1,b|c+1}\\ + M_{b,a,c|c} + M_{b,a,c+1|c+1} + M_{c,a,b|c} + M_{c+1,a,b|c+1}.
	        \end{aligned}
	        \end{equation*}							
			\item Suppose $a = c+1$.
			\begin{enumerate}[label=\theenumi.\theenumii.\arabic*)]
				\item If $b = c-1$, then
			        \begin{equation*}
			        \begin{aligned}            
			        P_{c-1,c,c+1|c} =  M_{c-1,c,c+1|c} + M_{c-1,c+1,c|c} + M_{c-1,c+1,c+1|c+1} \\ 
                 + M_{c,c-1,c+1|c} + M_{c+1,c-1,c|c} + M_{c+1,c-1,c+1|c+1} \\
                 + M_{c,c+1,c-1|c} + M_{c+1,c,c-1|c} + M_{c+1,c+1,c-1|c+1}.
			        \end{aligned}
			        \end{equation*}				
				\item If $b < c-1$, then
	        		\begin{equation*}
			        \begin{aligned}            
			        P_{b,c,c+1|c} =  M_{b,c,c+1|c} + M_{b,c+1,c|c} + M_{c,b,c+1|c} \\
                 + M_{c,c+1,b|c} + M_{c+1,b,c|c} + M_{c+1,b,c|c}.
                	 \end{aligned}
			        \end{equation*}					
			\end{enumerate}					
		\item Suppose $a < c$ and $a > b$.
	        \begin{equation*}
	        \begin{aligned}            
	        P_{b,c,a|c} = M_{b,c,a|c} + M_{b,c+1,a|c+1} + M_{c,b,a|c} + M_{c+1,b,a|c+1}\\
	                    + M_{a,c,b|c} + M_{a,c+1,b|c+1} + M_{c,a,b|c} + M_{c+1,a,b|c+1}.
	        \end{aligned}
	        \end{equation*}		
	\end{enumerate}		
\end{enumerate}
\end{thm}
\begin{thm}\label{gl315}
Let $a, b, c \in \mathbb{Z}$ with $a \geq b$. The projective objects $P_{c, a, b| c}$ in $\Bl_{a,b}$ have the following Verma flag formulae.
\begin{enumerate}[label=(\arabic*), ref=(\arabic*]
	\item Suppose $b > c$.
	\begin{enumerate}[label=\theenumi.\arabic*), ref = \arabic*]
		\item Suppose $b > c+1$.
		\begin{enumerate}[label=\theenumi.\theenumii.\arabic*)]
			\item If $a > b$, then
	        \begin{equation*}
	        \begin{aligned}            
                  P_{c,a,b|c} = M_{c,a,b|c} + M_{c+1,a,b|c+1} + M_{a,c,b|c} + M_{a,c+1,b|c+1} \\
                  M_{b,a,c|c} + M_{b,a,c+1|c+1} + M_{a,b,c|c} + M_{a,b,c+1|c+1}.
	        \end{aligned}
	        \end{equation*}			
			\item If $a = b$, then
	        \begin{equation*}
	        \begin{aligned}
	        P_{c,b,b|c} = M_{c,b,b|c} + M_{c+1,b,b|c+1} + M_{b,c,b|c} \\
                + M_{b,c+1,b|c+1} + M_{b,b,c|c} + M_{b,b,c+1|c+1}.
	        \end{aligned}                
	        \end{equation*}				
		\end{enumerate}	
        
		\item Suppose $b = c+1$.
		\begin{enumerate}[label=\theenumi.\theenumii.\arabic*)]
			\item If $a > c+1$, then
	        \begin{equation*}
	        \begin{aligned}            
                 P_{c,a,c+1|c} = M_{c,a,c+1|c} + M_{c+1,a,c|c} + M_{c+1,a,c+1|c+1} \\
                + M_{a,c,c+1|c} + M_{a,c+1,c|c} + M_{a,c+1,c+1|c+1}.
	        \end{aligned}
	        \end{equation*}			
			\item If $a = c+1$, then
	        \begin{equation*}
	        \begin{aligned}
	        P_{c,c+1,c+1|c} = M_{c,c+1,c+1|c} + M_{c+1,c,c+1|c}\\ + M_{c+1,c+1,c|c} + M_{c+1,c+1,c+1|c+1}.
	        \end{aligned}	        
	        \end{equation*}				
						
		\end{enumerate}	
	\end{enumerate}
	\item Suppose $b < c$.
	\begin{enumerate}[label=\theenumi.\arabic*), ref = \arabic*]
		\item Suppose $a > c$.
		\begin{enumerate}[label=\theenumi.\theenumii.\arabic*)]
			\item If $a > c+1$, then
	        \begin{equation*}
	        \begin{aligned}            
	        P_{c,a,b|c} = M_{c,a,b|c} + M_{c+1,a,b|c+1} + M_{a,c,b|c} + M_{a,c+1,b|c+1}.
	        \end{aligned}	        
	        \end{equation*}							
			\item If $a = c+1$, then
	        \begin{equation*}
	        \begin{aligned}            
	        P_{c,c+1,b|c} = M_{c,c+1,b|c} + M_{c+1,c+1,b|c+1}.
	        \end{aligned}
	        \end{equation*}				
		\end{enumerate}	
		\item If $a < c$, then
	        \begin{equation*}
	        \begin{aligned}            
	        P_{c,a,b|c} =M_{c,a,b|c} + M_{c+1,a,b|c+1}.
	        \end{aligned}
	        \end{equation*}				
			
		\end{enumerate}		
\end{enumerate}
\end{thm}
\begin{thm}\label{gl316}
Let $a, b, c \in \mathbb{Z}$ with $a \geq b$. The projective objects $P_{c, b, a| c}$ in $\Bl_{a,b}$ have the following Verma flag formulae.
\begin{enumerate}[label=(\arabic*), ref=(\arabic*]
	\item Suppose $b > c$.
	\begin{enumerate}[label=\theenumi.\arabic*), ref = \arabic*]
		\item If $b > c+1$ and $a > b$, then
	        \begin{equation*}
	        \begin{aligned}            
                  P_{c,b,a|c} =  M_{c,b,a|c} + M_{c+1,b,a|c+1} + M_{b,c,a|c} 
                  M_{b,c+1,a|c+1} \\+ M_{c,a,b|c} + M_{c+1,a,b|c+1}
                  M_{b,a,c|c} + M_{b,a,c+1|c+1}\\ + M_{a,c,b|c}
                  M_{a,c+1,b|c+1} + M_{a,b,c|c} + M_{a,b,c+1|c+1}.
	        \end{aligned}
	        \end{equation*}							
        
		\item If $b = c+1$ and $a>c+1$,			
	        \begin{equation*}
	        \begin{aligned}            
                 P_{c,c+1,a|c} = M_{c,c+1,a|c} + M_{c+1,c,a|c} + M_{c+1,c+1,a|c+1} \\
                  M_{c,a,c+1|c} + M_{c+1,a,c|c} + M_{c+1,a,c+1|c+1} \\
                  M_{a,c,c+1|c} + M_{a,c+1,c|c} + M_{a,c+1,c+1|c+1}.
	        \end{aligned}
	        \end{equation*}							
	\end{enumerate}
	\item Suppose $b < c$.
	\begin{enumerate}[label=\theenumi.\arabic*), ref = \arabic*]
		\item Suppose $a > c$.
		\begin{enumerate}[label=\theenumi.\theenumii.\arabic*)]
			\item If $a > c+1$, then
	        \begin{equation*}
	        \begin{aligned}            
	        P_{c,b,a|c} =   M_{c,b,a|c} + M_{c+1,b,a|c+1} + M_{c,a,b|c} + M_{c+1,a,b|c+1} \\
                  M_{a,b,c|c} + M_{a,b,c+1|c+1} + M_{a,c,b|c} + M_{a,c+1,b|c+1}.
	        \end{aligned}	        
	        \end{equation*}							
			\item If $a = c+1$, then
	        \begin{equation*}
	        \begin{aligned}            
                  P_{c,b,c+1|c} =  M_{c,b,c+1|c} + M_{c,b,c|c} + M_{c+1,b,c+1|c+1} \\
                  M_{c,c+1,b|c} + M_{c,b,c|c} + M_{c+1,c+1,b|c+1}.
	        \end{aligned}
	        \end{equation*}				
		\end{enumerate}	
		\item If $a < c$ and $a > b$, then
	        \begin{equation*}
	        \begin{aligned}            
	        P_{c,b,a|c} = M_{c,b,a|c} + M_{c+1,b,a|c+1} + M_{c,a,b|c} + M_{c+1,a,b,c+1}
	        \end{aligned}.
	        \end{equation*}				
			
		\end{enumerate}		
\end{enumerate}
\end{thm}
\subsection{Proof}
In this subsection, we prove Theorems \ref{gl311} through \ref{gl316}. We use the method of translation functors by effecting certain representations. Here are the weights of these representations:
\\\\
Natural $V$:
    \begin{equation*}
    \begin{aligned}
    \{\delta_1, \ \delta_2, \ \delta_3, \ \epsilon \}    
    \end{aligned}
    \end{equation*}
Dual $V^*$:
    \begin{equation*}
    \begin{aligned}
    \{-\delta_1, \ -\delta_2, \ -\delta_3, \ -\epsilon \}    
    \end{aligned}
    \end{equation*}
Wedge-squared of the natural $\extp^2 V$:
    \begin{equation*}
    \begin{aligned}
    \{\delta_1 + \delta_2, \ \delta_2 + \delta_3, \ \delta_1 + \delta_3, \\ \delta_1 + \epsilon, \ \delta_2 + \epsilon ,  \ \delta_3 + \epsilon \ ,  2\epsilon\}    
    \end{aligned}
    \end{equation*}
Wedge-cubed of the natural $\extp^3 V$:
    \begin{equation*}
    \begin{aligned}
    \{\delta_1 + \delta_2 + \delta_3, \ \delta_1 + \delta_2 + \epsilon, \ \delta_2 + \delta_3 + \epsilon, \\ \delta_1 + \delta_3 + \epsilon, \ \delta_1 + 2\epsilon ,  \ \delta_2 + 2\epsilon \ , \delta_3 + 2\epsilon, \  3\epsilon\}    
    \end{aligned}
    \end{equation*}
Wedge-squared of the dual $\extp^2 V^*$:
    \begin{equation*}
    \begin{aligned}
    \{-\delta_1 - \delta_2, \ -\delta_2 - \delta_3, \ -\delta_1 - \delta_3, \\ -\delta_1 - \epsilon, \ -\delta_2 - \epsilon ,  \ -\delta_3 - \epsilon \ ,  -2\epsilon\}    
    \end{aligned}
    \end{equation*}
We now offer justification for the formulae above, separated into cases that have different formulae, based on the strategy in \cref{strategy}. Our proof will be more explicit in the earlier cases and cases which require more sophisticated techniques; those which lack much explanation follow the strategy almost directly.
\subsubsection{Case: $P_\lambda, \ \lambda \cong (a,b,c \ | \ c)$ }
\begin{enumerate}[label=(\arabic*), ref=(\arabic*]
\item $b > c$.
    \begin{enumerate}[label=\theenumi.\arabic*), ref = \arabic*]
        \item $b > c+1$. \\ Let $\mu \coloneqq \lambda - \epsilon$. By Proposition \ref{typAndDom} for the Verma flag  of $P_\mu$ and Proposition \ref{sum} for the Verma flag of the projection, we have the following:
            \begin{equation*}
            \begin{aligned}  
            P_\mu = M_{a,b,c|c+1},
            \mathrm{Pr}_\lambda (P_\mu\otimes V)= M_{a,b,c|c} + M_{a,b,c+1|c+1}.
            \end{aligned}
            \end{equation*} 
        We have that $P_\lambda$ is a direct summand. Apply Corollary \ref{corlen} to deduce that it is the only summand.
        \item $b = c+1$.
        \begin{enumerate}[label=\theenumi.\theenumii.\arabic*)]
            \item $a > c+2$. \\ Use $P_\mu$ as indicated below to produce the following standard filtration for the projection.
                \begin{equation*}
                \begin{aligned}
                P_\mu = M_{a,c+1,c|c+2}, \\
                \mathrm{Pr}_\lambda (P_\mu\otimes \extp^2 V) = M_{a,c+1,c|c} + M_{a,c+1,c+1|c+1}\\ + M_{a,c+2,c+1|c+2}.
                \end{aligned}
                \end{equation*}
            Now, $P_\lambda$ is a direct summand of the projection. We now apply . The module $M_{a,c+1,c+1|c+1}$ appears in the Verma flag. By Proposition \ref{filprop}[\ref{p4}], $M_{a,c+2,c+1|c+2}$ appears. We deduce that $P_\lambda$ is the only direct summand.
            \item $a=c+2$. \\ Use $P_\mu$ as indicated below and proceed with the strategy in \cref{strategy} to see $P_\lambda$ is the only direct summand in the projection.
                \begin{equation*}
                \begin{aligned}
                P_\mu = M_{c+2,c+1,c|c+3}, \\
                \mathrm{Pr}_\lambda (P_\mu\otimes \extp^3 V) = M_{c+2,c+1,c|c} + M_{c+2,c+2,c+1|c+2}\\ + M_{c+2,c+1,c+1|c+1} + M_{c+3,c+2,c+1|c+3}.
                \end{aligned}
                \end{equation*}
            \item $a = c+1$. \\ Use $P_\mu$ as indicated below and proceed with the strategy in \cref{strategy} to see $P_\lambda$ is the only direct summand in the projection.
                \begin{equation*}
                \begin{aligned}
                P_\mu = M_{c+1,c+1,c|c+2}, \\
                \mathrm{Pr}_\lambda (P_\mu\otimes \extp^2 V) = M_{c+1,c+1,c|c} + M_{c+1,c+1,c+1|c+1} \\ + M_{c+1,c+2,c+1|c+2} + M_{c+2,c+1,c+1|c+2}.
                \end{aligned}
                \end{equation*}
        \end{enumerate}
    \end{enumerate}
    
\item $b = c$.
    \begin{enumerate} [label=\theenumi.\arabic*), ref = \arabic*]
        \item $a > c+1$. \\ Use $P_\mu$ as indicated below and proceed with the strategy in \cref{strategy} to see $P_\lambda$ is the only direct summand in the projection. 
            \begin{equation*}
            \begin{aligned}
            P_\mu = M_{a,c,c|c+1},\\
            \mathrm{Pr}_\lambda (P_\mu\otimes V) = M_{a,c,c|c} + M_{a,c,c+1|c+1} + M_{a,c+1,c|c+1}.
            \end{aligned}
            \end{equation*}
        \item $a = c+1$. \\Use $P_\mu$ as indicated below and proceed with the strategy in \cref{strategy} to see $P_\lambda$ is the only direct summand in the projection. 
            \begin{equation*}
            \begin{aligned}
            P_\mu = M_{c+1,c,c|c+2}, \\
            \mathrm{Pr}_\lambda (P_\mu\otimes \extp^2 V) = M_{c+1,c,c|c} + M_{c+2,c+1,c|c+2} \\ + M_{c+1,c,c+1|c+1} + M_{c+1,c+1,c|c+1} + M_{c+2,c,c+1|c+2}.
            \end{aligned}
            \end{equation*}
        \item $a = c$. \\ Use $P_\mu$ as indicated below and proceed with the strategy in \cref{strategy} to see $P_\lambda$ is the only direct summand in the projection. 
            \begin{equation*}
            \begin{aligned}
            P_\mu = M_{c,c,c|c+1},\\
            \mathrm{Pr}_\lambda (P_\mu\otimes V) = M_{c,c,c|c} + M_{c,c,c+1|c+1} \\ + M_{c,c+1,c|c+1} + M_{c+1,c,c|c+1}.
            \end{aligned}
            \end{equation*}

    \end{enumerate}

\item $b < c$
    \begin{enumerate} [label=\theenumi.\arabic*), ref = \arabic*]
        \item $a > c$.
        \begin{enumerate}[label=\theenumi.\theenumii.\arabic*)]
            \item $a > c+1$. \\ Use $P_\mu$ as indicated below and proceed with the strategy in \cref{strategy} to see $P_\lambda$ is the only direct summand in the projection. 
                \begin{equation*}
                \begin{aligned}
                P_\mu = M_{a,b,c|c+1} + M_{a,c,b|c+1}, \\
                \mathrm{Pr}_\lambda (P_\mu\otimes V) = M_{a,b,c|c} + M_{a,b,c+1|c+1}\\ + M_{a,c,b|c} + M_{a,c+1,b|c+1}.
                \end{aligned}
                \end{equation*}
            \item $a = c+1$. \\Use $P_\mu$ as indicated below and proceed with the strategy in \cref{strategy} to see $P_\lambda$ is the only direct summand in the projection. 
                \begin{equation*}
                \begin{aligned}
                P_\mu = M_{c+1,b,c|c+2} + M_{c+1,c,b|c+2}, \\
                \mathrm{Pr}_\lambda (P_\mu\otimes \extp^2 V) = M_{c+1,b,c|c} + M_{c+1,b,c+1|c+1} \\+ M_{c+2,b,c+1|c+2} + M_{c+1,c,b|c}\\ + M_{c+1,c+1,b|c+1} + M_{c+2,c+1,b|c+2}.
                \end{aligned}
                \end{equation*}

        \end{enumerate}
    
        \item $a = c$. \\Use $P_\mu$ as indicated below and proceed with the strategy in \cref{strategy} to see $P_\lambda$ is the only direct summand in the projection. 
            \begin{equation*}
            \begin{aligned}
            P_\mu = M_{c,b,c|c+1} + M_{c,c,b|c+1}, \\
            \mathrm{Pr}_\lambda (P_\mu\otimes V) = M_{c,b,c|c} + M_{c,b,c+1|c+1} + M_{c+1,b,c|c+1} \\ + M_{c,c,b|c} + M_{c,c+1,b|c+1} + M_{c+1,c,b|c+1}.
            \end{aligned}
            \end{equation*}
   
        \item $a < c$. 
        \begin{enumerate}[label=\theenumi.\theenumii.\arabic*)]
            \item $a > b$. \\ Use $P_\mu$ as indicated below and proceed with the strategy in \cref{strategy} to see $P_\lambda$ is the only direct summand in the projection. 
                \begin{equation*}
                \begin{aligned}
                P_\mu = M_{a,b,c|c+1} + M_{a,c,b|c+1} + M_{c,a,b|c+1}\\ + M_{c,b,a|c+1}, \\
                \mathrm{Pr}_\lambda (P_\mu\otimes V) = M_{a,b,c|c} + M_{a,b,c+1|c+1} + M_{a,c,b|c}\\  + M_{c,a,b|c} + M_{a,c+1,b|c+1} + M_{c+1,a,b|c+1}\\ + M_{c,b,a|c} + M_{c+1,b,a|c+1}.
                \end{aligned}
                \end{equation*}
            \item $a = b$. \\ Use $P_\mu$ as indicated below and proceed with the strategy in \cref{strategy} to see $P_\lambda$ is the only direct summand in the projection. 
                \begin{equation*}
                \begin{aligned}
                P_\mu = M_{b,b,c|c+1} + M_{b,c,b|c+1} + M_{c,b,b|c+1}, \\
                \mathrm{Pr}_\lambda (P_\mu\otimes V) = M_{b,b,c|c} + M_{b,b,c+1|c+1} + M_{b,c,b|c}\\  + M_{b,c+1,b|c+1}  + M_{c,b,b|c} + M_{c+1,b,b|c+1}.
                \end{aligned}
                \end{equation*}

        \end{enumerate}
    
    \end{enumerate}
\end{enumerate}

\subsubsection{Case: $P_\lambda, \ \lambda\cong(b,a,c\ | \ c)$}
\begin{enumerate}[label=(\arabic*), ref=(\arabic*]
\item $b > c$.
    \begin{enumerate}[label=\theenumi.\arabic*), ref = \arabic*]
        \item $b > c+1$.
        \begin{enumerate}[label=\theenumi.\theenumii.\arabic*)]
            \item $a > b$. \\ Use $P_\mu$ as indicated below and proceed with the strategy in \cref{strategy} to see $P_\lambda$ is the only direct summand in the projection.
                \begin{equation*}
                \begin{aligned}
                P_\mu = M_{b,a,c|c+1} + M_{a,b,c|c+1}, \\ 
                \mathrm{Pr}_\lambda (P_\mu \otimes V) = M_{b,a,c|c} + M_{b,a,c+1|c+1} + M_{a,b,c|c}\\ + M_{a,b,c+1|c+1}.
                \end{aligned}
                \end{equation*}
            \item $a = b$. \\ This is a repeat of an earlier case.
        \end{enumerate}    
        \item $b = c+1$.
        \begin{enumerate}[label=\theenumi.\theenumii.\arabic*)]
            \item $a > c+2$. \\  Use $P_\mu$ as indicated below and proceed with the strategy in \cref{strategy} to see $P_\lambda$ is the only direct summand in the projection.
                \begin{equation*}
                \begin{aligned}
                P_\mu = M_{c+1,a,c|c+2} +  M_{c+1,a,c|c+2}, \\
                \mathrm{Pr}_\lambda (P_\mu \otimes \extp^2 V) = M_{c+1,a,c|c} + M_{c+2,a,c+1|c+2}\\ + M_{c+1,a,c+1|c+1} + M_{a,c+1,c|c}\\ + M_{a,c+2,c+1|c+2}+ M_{a,c+1,c+1|c+1}.
                \end{aligned}
                \end{equation*}         
            \item $a = c+2$. \\  Use $P_\mu$ as indicated below and proceed with the strategy in \cref{strategy} to deduce the following:
                \begin{equation*}
                \begin{aligned}
                P_\mu = M_{c+1,c+2,c|c+3} +  M_{c+2,c+1,c|c+3}, \\
                P_1\coloneqq \mathrm{Pr}_\lambda (P_\mu \otimes \extp^3 V) = M_{c+1,c+2,c|c} + 2M_{c+2,c+2,c+1|c+2}\\ + M_{c+1,c+2,c+1|c+1} + M_{c+2,c+3,c+1|c+3}  + M_{c+2,c+1,c|c}\\ + M_{c+2,c+1,c+1|c+1} + M_{c+3,c+2,c+1|c+3}.
                \end{aligned}
                \end{equation*}
            If we proceed with our strategy, we realize that $M_{c+1,c+2,c|c}$, $M_{c+2,c+1,c|c}$, $M_{c+1,c+2,c+1|c+1}$, $M_{c+2,c+1,c+1|c+1}$, and one copy of $M_{c+2,c+2,c+1|c+2}$ are the only modules we deduce that can appear in the standard filtration of $P_\lambda$. This leaves three terms remaining, of which the one with lowest weight is the second copy of $M_{c+2,c+2,c+1|c+2}$. This means that if there were another summand, it would be $P_{c+2,c+2,c+1|c+2}$. Observe that this is of the form $P_{\hat{c}, \hat{c}, \hat{b}|\hat{c}}$, where $\hat{b}, \hat{c} \in \mathbb{Z}$ and $\hat{b} < \hat{c}$ (Theorem \ref{gl313}, Item 2.2). This has a Verma flag of length $3$, which would correspond to the remaining Verma modules. We realize that there may be another summand in the projection. Therefore, we try another approach.
            \\\\
            We now let $\theta \coloneqq \lambda - (\delta_2 + \epsilon)$. Observe that $\theta \cong (c+1, c+1,c\ |\ c+1)$ is atypical, so by looking at the same case (Theorem \ref{gl313}, Item 2.2) \ref{help0}), we have the following:
                \begin{equation*}
                \begin{aligned}
                P_\theta = M_{c+1,c+1,c|c+1} +  M_{c+1,c+2,c|c+2} + M_{c+2,c+1,c|c+2}, \\
                P_2\coloneqq\mathrm{Pr}_\lambda (P_\theta \otimes \extp^2 V) = 2M_{c+1,c+2,c|c} + 2M_{c+2,c+1,c|c} \\+ 2M_{c+1,c+2,c+1|c+1} + 2M_{c+2,c+1,c+1|c+1}\\ + 2M_{c+2,c+2,c+1|c+2}.
                \end{aligned}
                \end{equation*}   
            We have that $P_\lambda$ also appears as a direct summand for both $P_1$ and $P_2$. However, in a Verma flag for $P_2$, we have the same five terms each appear twice. We deduce that $P_\lambda$ cannot have more than $5$ terms and that both $P_1$ and $P_2$ split up as the direct sum of two projectives. We conclude that these same five terms form the Verma flag for $P_\lambda$.
            \item $a = c+1$. \\ This is a repeat of an earlier case.
        \end{enumerate}

    \end{enumerate}
\item $b = c$.
    \begin{enumerate}[label=\theenumi.\arabic*), ref = \arabic*]
        \item $a > c$. 
        \begin{enumerate}[label=\theenumi.\theenumii.\arabic*)]
            \item $a > c+1$. \\ Use $P_\mu$ as indicated below and proceed with the strategy in \cref{strategy} to see $P_\lambda$ is the only direct summand in the projection.
                \begin{equation*}
                \begin{aligned}
                P_\mu = M_{c,a,c|c+1} +  M_{a,c,c|c+1}, \\
                \mathrm{Pr}_\lambda (P_\mu \otimes V) = M_{c,a,c|c} + M_{c,a,c+1|c+1} + M_{c+1,a,c|c+1} \\ + M_{a,c,c|c} + M_{a,c,c+1|c+1} + M_{a,c+1,c|c+1}.
                \end{aligned}
                \end{equation*}    
    
            \item $a = c+1$. \\ Use $P_\mu$ as indicated below and proceed with the strategy in \cref{strategy} to see $P_\lambda$ is the only direct summand in the projection.
                \begin{equation*}
                \begin{aligned}
                P_\mu = M_{c,c+1,c|c+2} +  M_{c+1,c,c|c+2}, \\
                \mathrm{Pr}_\lambda (P_\mu \otimes \extp^3 V^*) = M_{c,c+1,c|c} + M_{c,c+2,c+1|c+2} + M_{c+1,c+2,c|c+2} \\ + M_{c,c+1,c+1|c+1} + 2M_{c+1,c+1,c|c+1} + M_{c+2,c,c+1|c+2}\\ + M_{c+2,c+1,c|c+2}  + M_{c+1,c,c+1|c+1} + M_{c+1,c,c|c}.
                \end{aligned}
                \end{equation*}
        \end{enumerate}
        \item $a = c$. \\ This is a repeat of an earlier case.   
    \end{enumerate}
\item $b < c$.
    \begin{enumerate}[label=\theenumi.\arabic*), ref = \arabic*]
    \item $a > c$.
    \begin{enumerate}[label=\theenumi.\theenumii.\arabic*)]
        \item $a > c+1$. \\ Use $P_\mu$ as indicated below and proceed with the strategy in \cref{strategy} to see $P_\lambda$ is the only direct summand in the projection.
            \begin{equation*}
            \begin{aligned}
            P_\mu = M_{b,a,c|c+1} + M_{a,b,c|c+1} + M_{a,c,b|c+1}\\ + M_{c,a,b|c+1}, \\
            \mathrm{Pr}_\lambda (P_\mu \otimes V) = M_{b,a,c|c} + M_{a,b,c|c} + M_{a,c,b|c} + M_{c,a,b|c}\\ + M_{b,a,c+1|c+1} + M_{a,b,c+1|c+1} + M_{a,c+1,b|c+1}\\ + M_{c+1,a,b|c+1}.
            \end{aligned}
            \end{equation*}          
        \item $a=c+1$. \\ Use $P_\mu$ as indicated below and proceed with the strategy in \cref{strategy} to see $P_\lambda$ is the only direct summand in the projection.
            \begin{equation*}
            \begin{aligned}
            P_\mu = M_{b,c+1,c|c+2} + M_{c+1,b,c|c+2} + M_{c+1,c,b|c+2}\\ + M_{c+1,a,b|c+2}, \\
            \mathrm{Pr}_\lambda (P_\mu \otimes \extp^2 V) = 
              M_{b,c+1,c|c} + M_{b,c+2,c+1|c+2} + M_{b,c+1,c+1|c+1} \\
            + M_{c+1,b,c|c} + M_{c+2,b,c+1|c+2} + M_{c+1,b,c+1|c+1} \\
            + M_{c,c+1,b|c} + M_{c+1,c+2,b|c+2} + M_{c+1,c+1,b|c+1} \\
            + M_{c+1,c,b|c} + M_{c+2,c+1,b|c+2} + M_{c+1,c+1,b|c+1}.
            \end{aligned}
            \end{equation*}   
        \end{enumerate}    
        \item $a = c$. \\ Use $P_\mu$ as indicated below and proceed with the strategy in \cref{strategy} to see $P_\lambda$ is the only direct summand in the projection.
            \begin{equation*}
            \begin{aligned}         
            P_\mu = M_{b,c,c|c+1} + M{c,b,c|c+1} + M_{c,c,b|c+1}, \\
            \mathrm{Pr}_\lambda (P_\mu \otimes V) =
              M_{b,c,c|c} + M_{b,c,c+1|c+1} + M_{b,c+1,c|c+1} \\
            + M_{c,b,c|c} + M_{c,b,c+1|c+1} + M_{c+1,b,c|c+1} \\
            + M_{c,c,b|c} + M_{c,c+1,b|c+1} + M_{c+1,c,b|c+1}.
            \end{aligned}
            \end{equation*}
        \item $a < c$.
        \begin{enumerate}[label=\theenumi.\theenumii.\arabic*)]
            \item $a > b$. \\ Use $P_\mu$ as indicated below and proceed with the strategy in \cref{strategy} to see $P_\lambda$ is the only direct summand in the projection.
            \begin{equation*}
            \begin{aligned}   
            P_\mu = M_{b,a,c|c+1} + M_{b,c,a|c+1} + M_{a,b,c|c+1} + M_{c,b,a|c+1}\\ + M_{a,c,b|c+1} + M_{c,a,b|c+1}, \\
            \mathrm{Pr}_\lambda (P_\mu \otimes V) = 
              M_{b,a,c|c} + M_{b,a,c+1|c+1} + M_{b,c,a|c} \\
            + M_{b,c+1,a|c+1} + M_{a,b,c|c} + M_{a,b,c+1|c+1} \\
            + M_{c,b,a|c} + M_{c+1,b,a|c+1} + M_{a,c,b|c} \\
            + M_{a,c+1,b|c+1} + M_{c,a,b|c} + M_{c+1,a,b|c+1}.
            \end{aligned}
            \end{equation*}             
            \item $a = b$. \\ This is a repeat of an earlier case.
        \end{enumerate}
    \end{enumerate}
\end{enumerate}

\subsubsection{Case: $P_\lambda, \ \lambda \cong  (a,c,b \ | \ c)$ }
\begin{enumerate}[label=(\arabic*), ref=(\arabic*]
\item $b > c$.
    \begin{enumerate}[label=\theenumi.\arabic*), ref = \arabic*]
        \item $b > c+1$. \\  Use $P_\mu$ as indicated below and proceed with the strategy in \cref{strategy} to see $P_\lambda$ is the only direct summand in the projection.
            \begin{equation*}
            \begin{aligned}
            P_\mu = M_{a,c,b|c+1} + M_{a,b,c|c+1}, \\
            \mathrm{Pr}_\lambda (P_\mu \otimes V) = M_{a,c,b|c} + M_{a,c+1,b|c+1} + M_{a,b,c|c} + M_{a,b,c+1|c+1}.
            \end{aligned}
            \end{equation*}      
        \item $b = c+1$
        \begin{enumerate}[label=\theenumi.\theenumii.\arabic*)]
            \item $a > c+1$. \\ Use $P_\mu$ as indicated below and proceed with the strategy in \cref{strategy} to see $P_\lambda$ is the only direct summand in the projection.
                \begin{equation*}
                \begin{aligned}
                P_\mu = M_{a+1,c+1,c+1|c}, \\
                \mathrm{Pr}_\lambda (P_\mu \otimes \extp^2 V^*) = M_{a,c,c+1|c} + M_{a,c+1,c+1|c+1} + M_{a,c+1,c|c}.
                \end{aligned}
                \end{equation*}    
            \item $a = c+1$. \\Use $P_\mu$ as indicated below and proceed with the strategy in \cref{strategy} to see $P_\lambda$ is the only direct summand in the projection.
                \begin{equation*}
                \begin{aligned}
                P_\mu = M_{c+2,c+1,c+1|c}, \\
                \mathrm{Pr}_\lambda (P_\mu \otimes \extp^2 V^*) = M_{c+1,c,c+1|c}\\ + M_{c+1,c+1,c+1|c+1} + M_{c+2,c+1,c+1|c+2} + M_{c+1,c+1,c|c}.
                \end{aligned}
                \end{equation*}                
                        
        \end{enumerate}
    
    \end{enumerate}

\item $b = c$. This is a repeat of an earlier case.
\item $b < c$.
    \begin{enumerate}[label=\theenumi.\arabic*), ref = \arabic*]
    \item $a > c$.
     \begin{enumerate}[label=\theenumi.\theenumii.\arabic*)]
        \item $a > c+1$. \\Use $P_\mu$ as indicated below and proceed with the strategy in \cref{strategy} to see $P_\lambda$ is the only direct summand in the projection.
            \begin{equation*}
            \begin{aligned}        
            P_\mu = M_{a,c,b|c+1}, \\
            \mathrm{Pr}_\lambda (P_\mu \otimes V) = M_{a,c,b|c} + M_{a,c+1,b|c+1}.
            \end{aligned}
            \end{equation*}        
        \item $a = c+1$. \\Use $P_\mu$ as indicated below and proceed with the strategy in \cref{strategy} to see $P_\lambda$ is the only direct summand in the projection.
            \begin{equation*}
            \begin{aligned}
            P_\mu = M_{c+1,c,b|c+2}, \\
            \mathrm{Pr}_\lambda (P_\mu \otimes \extp^2 V) = M_{c+1,c,b|c} + M_{c+1,c+1,b|c+1} + M_{c+2,c+1,b|c+2}.
            \end{aligned}
            \end{equation*}   
        \end{enumerate}
        \item $a = c$. \\ Use $P_\mu$ as indicated below and proceed with the strategy in \cref{strategy} to see $P_\lambda$ is the only direct summand in the projection.
            \begin{equation*}
            \begin{aligned}  
            P_\mu = M_{c,c,b|c+1}, \\
            \mathrm{Pr}_\lambda (P_\mu \otimes V) = M_{c,c,b|c} + M_{c,c+1,b|c+1} + M_{c+1,c,b|c+1}.
            \end{aligned}
            \end{equation*}  
        \item $a < c$. \\  Use $P_\mu$ as indicated below and proceed with the strategy in \cref{strategy} to see $P_\lambda$ is the only direct summand in the projection.
            \begin{equation*}
            \begin{aligned}
            P_\mu = M_{a,c,b|c+1} + M_{c,a,b|c+1}, \\
            \mathrm{Pr}_\lambda (P_\mu \otimes V) = M_{a,c,b|c} + M_{a,c+1,b|c+1} + M_{c,a,b|c} + M_{c+1,a,b|c+1}.
            \end{aligned}
            \end{equation*}
    \end{enumerate}

\end{enumerate}

\subsubsection{Case: $P_\lambda, \ \lambda \cong  (b,c,a \ | \ c)$}
\begin{enumerate}[label=(\arabic*), ref=(\arabic*]
    \item $b > c$.
    \begin{enumerate}[label=\theenumi.\arabic*), ref = \arabic*]
        \item $b > c+1$.
        \begin{enumerate}[label=\theenumi.\theenumii.\arabic*)]
            \item $a > b$. \\ Use $P_\mu$ as indicated below and proceed with the strategy in \cref{strategy} to see $P_\lambda$ is the only direct summand in the projection.
            \begin{equation*}
            \begin{aligned}
            P_\mu = M_{b,c,a|c+1} + M_{b,a,c|c+1} + M_{a,b,c|c+1} + M_{a,c,b|c+1}, \\
            \mathrm{Pr}_\lambda (P_\mu \otimes V) = 
              M_{b,c,a|c} + M_{b,a,c|c} + M_{a,b,c|c} + M_{a,c,b|c}\\
            + M_{b,c+1,a|c+1} + M_{b,a,c+1|c+1} + M_{a,b,c+1|c+1} + M_{a,c+1,b|c+1}.
            \end{aligned}
            \end{equation*}         
            \item $a = b$. \\ This is a repeat of an earlier case.
        \end{enumerate}    
        \item $b = c+1$.
        \begin{enumerate}[label=\theenumi.\theenumii.\arabic*)]
            \item $a > c+2$. \\  Use $P_\mu$ as indicated below and proceed with the strategy in \cref{strategy} to see $P_\lambda$ is the only direct summand in the projection.
            \begin{equation*}
            \begin{aligned}
            P_\mu = M_{c+1,c,a|c+2} + M_{a,c,c+1|c+2} + M_{c+1,a,c|c+2} + M_{a,c+1,c|c+2}, \\
            \mathrm{Pr}_\lambda (P_\mu \otimes \extp^2 V) = 
              M_{c+1,c,a|c} + M_{c+2,c+1,a|c+2} + M_{c+1,c+1,a|c+1} \\
            + M_{a,c,c+1|c} + M_{a,c+1,c+2|c+2} + M_{a,c+1,c+1|c+1} \\
            + M_{c+1,a,c|c} + M_{c+1,a,c+2|c+2} + M_{c+1,a,c+1|c+1} \\
            + M_{a,c+1,c|c} + M_{a,c+2,c+1|c+2} + M_{a,c+1,c+1|c+1}. \\
            \end{aligned}
            \end{equation*}         
            \item $a = c+2$. \\ Use $P_\mu$ as indicated below and proceed with the strategy in \cref{strategy} to see $P_\lambda$ is the only direct summand in the projection.
            \begin{equation*}
            \begin{aligned}
            P_\mu = M_{c+2,c+1,c+2|c} + M_{c+2,c+2,c+1|c}, \\
            \mathrm{Pr}_\lambda (P_\mu \otimes \extp^2 V^*) = 
              M_{c+1,c,c+2|c} + M_{c+2,c,c+1|c} + M_{c+1,c+1,c+2|c+1}\\ + M_{c+2,c+1,c+1|c+1}
            + M_{c+2,c+1,c+2|c+2} + M_{c+2,c+1,c|c}\\ + M_{c+1,c+2,c|c} + M_{c+1,c+2,c+1|c+1} 
            + M_{c+2,c+1,c+1|c+1}\\ + M_{c+2,c+2,c+1|c+2}.
            \end{aligned}
            \end{equation*}              
            \item $a = c+1$. \\ This is a repeat of an earlier case.

        \end{enumerate}

    \end{enumerate}
    \item $b = c$.
    \begin{enumerate}[label=\theenumi.\arabic*), ref = \arabic*]
        \item $a > c+1$. \\ Use $P_\mu$ as indicated below and proceed with the strategy in \cref{strategy} to see $P_\lambda$ is the only direct summand in the projection.
            \begin{equation*}
            \begin{aligned}
            P_\mu = M_{c,c,a|c+1} + M_{c,a,c|c+1} + M_{a,c,c|c+1}, \\
            \mathrm{Pr}_\lambda (P_\mu \otimes V) = 
              M_{c,c,a|c} + M_{c+1,c,a|c+1} + M_{c,c+1,a|c+1} \\
            + M_{c,a,c|c} + M_{c+1,a,c|c+1} + M_{c,a,c+1|c+1} \\
            + M_{a,c,c|c} + M_{a,c+1,c|c+1} + M_{a,c,c+1|c+1}.
            \end{aligned}
            \end{equation*}    
        \item $a = c+1$. \\ Use $P_\mu$ as indicated below and proceed with the strategy in \cref{strategy} to see $P_\lambda$ is the only direct summand in the projection.
            \begin{equation*}
            \begin{aligned}
            \mathrm{Pr}_\lambda (P_\mu \otimes \extp^2 V^*) =
             P_\mu = M_{c+1,c+1,c+1|c}, \\
              M_{c,c,c+1|c} + M_{c+1,c,c|c} + M_{c,c+1,c|c} \\
            + M_{c,c+1,c+1|c+1} + M_{c+1,c,c+1|c+1} + M_{c+1,c+1,c|c+1}.
            \end{aligned}
            \end{equation*}
        \item $a = c$. \\ This is a repeat of an earlier case. 
            
    \end{enumerate}    
    \item $b < c$.
    \begin{enumerate}[label=\theenumi.\arabic*), ref = \arabic*]
        \item $a > c+1$. \\ Use $P_\mu$ as indicated below and proceed with the strategy in \cref{strategy} to see $P_\lambda$ is the only direct summand in the projection.
            \begin{equation*}
            \begin{aligned}
            P_\mu = M_{b,c,a|c+1} + M_{c,b,a|c+1} + M_{a,b,c|c+1}\\
	        			+ M_{a,c,b|c+1} + M_{b,a,c|c+1} + M_{c,a,b|c+1},\\
             \mathrm{Pr}_\lambda (P_\mu \otimes V) = M_{b,c,a|c} + M_{b,c+1,a|c+1} + M_{c,b,a|c} + M_{c+1,b,a|c+1}\\ + M_{a,b,c|c} + M_{a,b,c+1|c+1}
	        			+ M_{a,c,b|c} + M_{a,c+1,b|c+1}\\ + M_{b,a,c|c} + M_{b,a,c+1|c+1} + M_{c,a,b|c} + M_{c+1,a,b|c+1}.
            \end{aligned}
            \end{equation*} 
        \item $a = c+1$. 
        \begin{enumerate}[label=\theenumi.\theenumii.\arabic*)]
            \item $b = c - 1$. \\ Use $P_\mu$ as indicated below and proceed with the strategy in \cref{strategy} to see $P_\lambda$ is the only direct summand in the projection.
                \begin{equation*}
                \begin{aligned} 
                P_\mu = M_{c-1,c+1,c+1|c} + M_{c+1,c-1,c+1|c} + M_{c+1,c+1,c-1|c}, \\
                \mathrm{Pr}_\lambda (P_\mu \otimes V^*) = 
                   M_{c-1,c,c+1|c} + M_{c-1,c+1,c|c} + M_{c-1,c+1,c+1|c+1} \\ 
                 + M_{c,c-1,c+1|c} + M_{c+1,c-1,c|c} + M_{c+1,c-1,c+1|c+1} \\
                 + M_{c,c+1,c-1|c} + M_{c+1,c,c-1|c} + M_{c+1,c+1,c-1|c+1}. 
                \end{aligned}
                \end{equation*}         
            \item $b < c - 1$. \\  Use $P_\mu$ as indicated below and proceed with the strategy in \cref{strategy} to see $P_\lambda$ is the only direct summand in the projection.
                \begin{equation*}
                \begin{aligned}            
                P_\mu = M_{b-1,c,c+1|c} + M_{b-1,c+1,c|c} + M_{c,b-1,c+1|c}\\ + M_{c+1,b-1,c|c} + M_{c,c+1,b-1|c} + M_{c+1,c,b-1|c}, \\
                \mathrm{Pr}_\lambda (P_\mu \otimes V) = 
                   M_{b,c,c+1|c} + M_{b,c+1,c|c} + M_{c,b,c+1|c} \\
                 + M_{c,c+1,b|c} + M_{c+1,b,c|c} + M_{c+1,b,c|c}.
                \end{aligned}
                \end{equation*}              
        \end{enumerate}            
        \item $a = c$. \\ This is a repeat of an earlier case.
        \item $a < c$.
        \begin{enumerate}[label=\theenumi.\theenumii.\arabic*)]
            \item $a > b$. \\ Use $P_\mu$ as indicated below and proceed with the strategy in \cref{strategy} to see $P_\lambda$ is the only direct summand in the projection.
                \begin{equation*}
                \begin{aligned}            
                P_\mu = M_{b,c,a|c+1} + M_{c,b,a|c+1} + M_{a,c,b|c+1} + M_{c,a,b|c+1}, \\
                \mathrm{Pr}_\lambda (P_\mu \otimes V) = 
                        M_{b,c,a|c} + M_{b,c+1,a|c+1} + M_{c,b,a|c} + M_{c+1,b,a|c+1}\\
	                    + M_{a,c,b|c} + M_{a,c+1,b|c+1} + M_{c,a,b|c} + M_{c+1,a,b|c+1}.
                \end{aligned}
                \end{equation*}              
            \item $a = b$. \\This is a repeat of an earlier case.
        \end{enumerate}
    \end{enumerate}
\end{enumerate}

\subsubsection{Case: $P_\lambda, \ \lambda \cong (c,a,b \ | \ c)$}
\begin{enumerate}[label=(\arabic*), ref=(\arabic*]
\item $b > c$.
    \begin{enumerate}[label=\theenumi.\arabic*), ref = \arabic*]
        \item $b > c+1$.
        \begin{enumerate}[label=\theenumi.\theenumii.\arabic*)]
            \item $a > b$. \\ Use $P_\mu$ as indicated below and proceed with the strategy in \cref{strategy} to see $P_\lambda$ is the only direct summand in the projection.
                \begin{equation*}
                \begin{aligned}
                P_\mu = M_{c+1,a,b|c} + M_{a,c+1,b|c} + M_{b,a,c+1|c} + M_{a,b,c+1|c}, \\
                \mathrm{Pr}_\lambda (P_\mu \otimes V^*) = 
                  M_{c,a,b|c} + M_{c+1,a,b|c+1} + M_{a,c,b|c} + M_{a,c+1,b|c+1} \\
                  M_{b,a,c|c} + M_{b,a,c+1|c+1} + M_{a,b,c|c} + M_{a,b,c+1|c+1}.
                \end{aligned}
                \end{equation*}         
            \item $a = b$. \\ Use $P_\mu$ as indicated below and proceed with the strategy in \cref{strategy} to see $P_\lambda$ is the only direct summand in the projection.
                \begin{equation*}
                \begin{aligned}
                P_\mu = M_{c+1,b,b|c} + M_{b,c+1,b|c} + M_{b,b,c+1|c}, \\
                \mathrm{Pr}_\lambda (P_\mu \otimes V^*) = 
                  M_{c,b,b|c} + M_{c+1,b,b|c+1} + M_{b,c,b|c} \\
                + M_{b,c+1,b|c+1} + M_{b,b,c|c} + M_{b,b,c+1|c+1}.
                \end{aligned}
                \end{equation*}

        \end{enumerate}    
        \item $b = c+1$.
        \begin{enumerate}[label=\theenumi.\theenumii.\arabic*)]
            \item $a > c+1$. \\ Use $P_\mu$ as indicated below and proceed with the strategy in \cref{strategy} to see $P_\lambda$ is the only direct summand in the projection.
                \begin{equation*}
                \begin{aligned}
                P_\mu = M_{c+1,a,c+1|c} + M_{a,c+1,c+1|c}, \\
                \mathrm{Pr}_\lambda (P_\mu \otimes V^*) = 
                  M_{c,a,c+1|c} + M_{c+1,a,c|c} + M_{c+1,a,c+1|c+1} \\
                + M_{a,c,c+1|c} + M_{a,c+1,c|c} + M_{a,c+1,c+1|c+1}.
                \end{aligned}
                \end{equation*}           
            \item $a = c+1$. \\ Use $P_\mu$ as indicated below and proceed with the strategy in \cref{strategy} to see $P_\lambda$ is the only direct summand in the projection.
                \begin{equation*}
                \begin{aligned}
                P_\mu = M_{c+1,c+1,c+1|c}, \\
                \mathrm{Pr}_\lambda (P_\mu \otimes V^*) = M_{c,c+1,c+1|c} + M_{c+1,c,c+1|c}\\ + M_{c+1,c+1,c|c} + M_{c+1,c+1,c+1|c+1}.
                \end{aligned}
                \end{equation*} 
        \end{enumerate}
    \end{enumerate}

\item $b = c$. This is a repeat of an earlier case.

\item $b < c$. 
    \begin{enumerate}[label=\theenumi.\arabic*), ref = \arabic*]
        \item $a > c$. 
        \begin{enumerate}[label=\theenumi.\theenumii.\arabic*)]
            
            \item $a > c+1$ \\ Use $P_\mu$ as indicated below and proceed with the strategy in \cref{strategy} to see $P_\lambda$ is the only direct summand in the projection. 
                \begin{equation*}
                \begin{aligned}
                P_\mu = M_{c+1,a,b|c} + M_{a,c+1,b|c}, \\ 
                \mathrm{Pr}_\lambda (P_\mu \otimes V^*) = M_{c,a,b|c} + M_{c+1,a,b|c+1} + M_{a,c,b|c} + M_{a,c+1,b|c+1}.
                \end{aligned}
                \end{equation*}        
            \item $a = c+1$ \\  Use $P_\mu$ as indicated below and proceed with the strategy in \cref{strategy} to see $P_\lambda$ is the only direct summand in the projection. 
                \begin{equation*}
                \begin{aligned}
                P_\mu = M_{c+1,c+1,b|c}, \\
                \mathrm{Pr}_\lambda (P_\mu \otimes V^*) = M_{c,c+1,b|c} + M_{c+1,c+1,b|c+1}.
                \end{aligned}
                \end{equation*}            

        \end{enumerate}    
        \item $a = c$. \\ This is a repeat of an earlier case.
         \item $a < c$. \\  Use $P_\mu$ as indicated below and proceed with the strategy in \cref{strategy} to see $P_\lambda$ is the only direct summand in the projection. 
            \begin{equation*}
            \begin{aligned}
            P_\mu = M_{c,a,b|c+1}, \\
            \mathrm{Pr}_\lambda (P_\mu \otimes V) = M_{c,a,b|c} + M_{c+1,a,b|c+1}.
            \end{aligned}
            \end{equation*}           

    \end{enumerate}
\end{enumerate}

\subsubsection{Case: $P_\lambda, \ \lambda \cong (c,b,a \ | \ c)$}
\begin{enumerate}[label=(\arabic*), ref=(\arabic*]
\item $b > c$.
    \begin{enumerate}[label=\theenumi.\arabic*), ref = \arabic*]
        \item $b > c+1$.
        \begin{enumerate}[label=\theenumi.\theenumii.\arabic*)]
            
            \item $a > b$. \\  Use $P_\mu$ as indicated below and proceed with the strategy in \cref{strategy} to see $P_\lambda$ is the only direct summand in the projection. 
                \begin{equation*}
                \begin{aligned}            
                P_\mu = M_{c,b,a|c+1} + M_{b,c,a|c+1} + M_{c,a,b|c+1} + M_{b,a,c|c+1}\\ + M_{a,c,b|c+1} + M_{a,b,c|c+1}, \\
                \mathrm{Pr}_\lambda (P_\mu \otimes V) = 
                  M_{c,b,a|c} + M_{c+1,b,a|c+1} + M_{b,c,a|c} \\
                  M_{b,c+1,a|c+1} + M_{c,a,b|c} + M_{c+1,a,b|c+1} \\
                  M_{b,a,c|c} + M_{b,a,c+1|c+1} + M_{a,c,b|c} \\
                  M_{a,c+1,b|c+1} + M_{a,b,c|c} + M_{a,b,c+1|c+1}.
                \end{aligned}
                \end{equation*}          
            \item $a = b$. \\ This is a repeat of an earlier case.       
           
        \end{enumerate}    
        \item $b = c+1$.
        \begin{enumerate}[label=\theenumi.\theenumii.\arabic*)]
            \item $a > c+1$. \\  Use $P_\mu$ as indicated below and proceed with the strategy in \cref{strategy} to see $P_\lambda$ is the only direct summand in the projection. 
                \begin{equation*}
                \begin{aligned}          
                P_\mu = M_{c+1,c+1,a|c} + M_{c+1,a,c+1|c} + M_{a,c+1,c+1}, \\
                \mathrm{Pr}_\lambda (P_\mu \otimes V^*) = 
                  M_{c,c+1,a|c} + M_{c+1,c,a|c} + M_{c+1,c+1,a|c+1} \\
                  M_{c,a,c+1|c} + M_{c+1,a,c|c} + M_{c+1,a,c+1|c+1} \\
                  M_{a,c,c+1|c} + M_{a,c+1,c|c} + M_{a,c+1,c+1|c+1}.
                \end{aligned}
                \end{equation*}         
            \item $a = c+1$. \\ This is the repeat of an earlier case.

        \end{enumerate}

    \end{enumerate}

\item $b = c$. \\ This is a repeat of an earlier case.

\item $b < c$.
    \begin{enumerate}[label=\theenumi.\arabic*), ref = \arabic*]
        \item $a > c$.
        \begin{enumerate}[label=\theenumi.\theenumii.\arabic*)]
            \item $a > c+1$. \\ Use $P_\mu$ as indicated below and proceed with the strategy in \cref{strategy} to see $P_\lambda$ is the only direct summand in the projection.  
                \begin{equation*}
                \begin{aligned}
                P_\mu = M_{c,b,a|c+1} + M_{c,a,b|c+1} + M_{a,b,c|c+1} + M_{a,b,c|c+1} + M_{a,c,b|c+1}, \\
                \mathrm{Pr}_\lambda (P_\mu \otimes V) = 
                  M_{c,b,a|c} + M_{c+1,b,a|c+1} + M_{c,a,b|c} + M_{c+1,a,b|c+1} \\
                  M_{a,b,c|c} + M_{a,b,c+1|c+1} + M_{a,c,b|c} + M_{a,c+1,b|c+1}.
                \end{aligned}
                \end{equation*}  
                
            \item $a = c+1$. \\  Use $P_\mu$ as indicated below and proceed with the strategy in \cref{strategy} to see $P_\lambda$ is the only direct summand in the projection.  
                \begin{equation*}
                \begin{aligned}
                P_\mu = M_{c+1,b,c+1|c} + M_{c+1,c+1,b|c}, \\
                \mathrm{Pr}_\lambda (P_\mu \otimes V^*) = 
                  M_{c,b,c+1|c} + M_{c,b,c|c} + M_{c+1,b,c+1|c+1} \\
                  M_{c,c+1,b|c} + M_{c,b,c|c} + M_{c+1,c+1,b|c+1}.
                \end{aligned}
                \end{equation*}              
           
        \end{enumerate}    
        \item $a = c$. \\ This is a repeat of an earlier case.
        \item $a < c$
        \begin{enumerate}[label=\theenumi.\theenumii.\arabic*)]
            \item $a > b$. \\  Use $P_\mu$ as indicated below and proceed with the strategy in \cref{strategy} to see $P_\lambda$ is the only direct summand in the projection. 
                \begin{equation*}
                \begin{aligned}
                P_\mu = M_{c,b,a|c+1} + M_{c,a,b|c+1}, \\                
                \mathrm{Pr}_\lambda (P_\mu \otimes V) = 
                  M_{c,b,a|c} + M_{c+1,b,a|c+1} + M_{c,a,b|c} + M_{c+1,a,b,c+1}.
                \end{aligned}
                \end{equation*}
                
            \item $a = b$. \\ This is a repeat of an earlier case.

        \end{enumerate}

    \end{enumerate}

\end{enumerate}

\section{Character Formulae for $\gl(2|2)$}\label{sec6}
In this section, we determine Verma multiplicities for standard filtration formulae for projective covers of simple modules of $\gl(2|2)$ with integral, atypical weight of degree $2$.
\subsection{Results}
Let $\g = \gl(2|2)$ have the standard choices of Cartan subalgebra, bilinear form, root system, positive, and fundamental system as described in \cref{prelims}.  Recall the notation described in \cref{link} to describe a weight in $\h^*$. Lastly, recall Example \ref{ex:gl22} and the corresponding block $\Bl_{0}$ (see \cref{blocks}). We have the following Theorems \ref{gl221} and \ref{gl222} that describe standard filtrations of projectives in this block.
\begin{thm}\label{gl221}
Let $a,b \in \mathbb{Z}$ with $a \geq b$. We have the following Verma flag formulae for the projective objects $P_{a, b|b, a}$ and $P_{a, b|a, b}$ in $\Bl_0$:
\begin{enumerate}[label=(\arabic*), ref=(\arabic*]
\item Case: $P_{a,b|b,a}$
\begin{enumerate}[label=\theenumi.\arabic*), ref = \arabic*]
    \item If $b < a - 1$, then
        \begin{equation*}
        \begin{aligned}            
        P_{a, b|b, a} = 
          M_{a,b|b,a} + M_{a+1,b|b,a+1} + M_{a,b+1|b+1,a} + M_{a+1,b+1|b+1,a+1}.
        \end{aligned}
        \end{equation*}
    \item If $b = a - 1$, then
        \begin{equation*}
        \begin{aligned}            
        P_{a,a-1|a-1,a}= 
          M_{a,a-1|a-1,a} + M_{a+1,a-1|a-1,a+1} + M_{a,a|a,a} \\
        + M_{a,a+1|a,a+1} + 2M_{a+1,a|a,a+1} + M_{a+1,a|a+1,a} \\
        + M_{a+1,a+1|a+1,a+1} + M_{a+2,a|a,a+2} + M_{a+2,a+1|a+1,a+2}.
        \end{aligned}
        \end{equation*}        
    \item If $b = a$, then
        \begin{equation*}
        \begin{aligned}      
        P_{a,a|a,a}= 
          M_{a,a|a,a} + M_{a+1,a|a,a+1} + M_{a,a+1|a+1,a} \\
        + M_{a+1,a+1|a+1,a+1} + M_{a+1,a|a+1,a} + M_{a,a+1|a,a+1}.
        \end{aligned} 
        \end{equation*}
\end{enumerate}
\item Case: $P_{a,b|a,b}$
\begin{enumerate}[label=\theenumi.\arabic*), ref = \arabic*]
    \item If $b < a - 1$, then
        \begin{equation*}
        \begin{aligned}            
        P_{a,b|a,b}= 
          M_{a,b|a,b} + M_{a+1,b|a+1,b} + M_{a,b+1|a,b+1} + M_{a+1,b+1|a+1,b+1} \\
        + M_{a,b|b,a} + M_{a+1,b|b,a+1} + M_{a,b+1|b+1,a} + M_{a+1,b+1|b+1,a+1}.
        \end{aligned}
        \end{equation*} 
    \item If $b = a - 1$, then
        \begin{equation*}
        \begin{aligned}            
        P_{a,a-1|a,a-1}  = 
          M_{a,a-1|a,a-1} + M_{a+1,a|a,a+1} + M_{a+1,a|a+1,a} + M_{a+1,a-1|a-1,a+1} \\
        + M_{a+1,a-1|a+1,a-1} + M_{a,a|a,a} + M_{a,a-1|a-1,a}.
        \end{aligned}
        \end{equation*}       
\end{enumerate}
\end{enumerate}
\end{thm}
\begin{thm}\label{gl222}
Let $a,b \in \mathbb{Z}$ with $a \geq b$. We have the following Verma flag formulae for the projective objects $P_{b, a|b, a}$ and $P_{b, a|a, b}$ in $\Bl_0$. 
\begin{enumerate}[label=(\arabic*), ref=(\arabic*]
\item Case: $P_{b,a|b,a}$
\begin{enumerate}[label=\theenumi.\arabic*), ref = \arabic*]
    \item If $b < a - 1$, then
        \begin{equation*}
        \begin{aligned}           
        P_{b,a|b,a} = 
           M_{b,a|b,a} + M_{b,a+1|b,a+1} + M_{b+1,a|b+1,a} + M_{b+1,a+1|b+1,a+1} \\
         + M_{a,b|b,a} + M_{a+1,b|b,a+1} + M_{a,b+1|b+1,a} + M_{a+1,b+1|b+1,a+1}
        \end{aligned}
        \end{equation*}  
    \item If $b = a - 1$, then
        \begin{equation*}
        \begin{aligned}            
        P_{a-1,a|a-1,a} = 
          M_{a-1,a|a-1,a} + M_{a,a-1|a-1,a} + M_{a,a|a,a} \\
        + M_{a-1,a+1|a-1,a+1} + M_{a,a+1|a,a+1} + M_{a+1,a-1|a-1,a+1} \\
        + M_{a+1,a|a,a+1}.
        \end{aligned}
        \end{equation*}       
\end{enumerate}
\item Case: $P_{b,a|a,b}$
\begin{enumerate}[label=\theenumi.\arabic*), ref = \arabic*]
    \item If $b < a - 1$, then
        \begin{equation*}
        \begin{aligned}            
        P_{b,a|a,b}  = 
           M_{b,a|a,b} + M_{b,a+1|a+1,b} + M_{b+1,a|a,b+1} + M_{b+1,a+1|a+1,b+1} \\
         + M_{b,a|b,a} + M_{b,a+1|b,a+1} + M_{b+1,a|b+1,a} + M_{b+1,a+1|b+1,a+1} \\   
         + M_{a,b|a,b} + M_{a+1,b|a+1,b} + M_{a,b+1|a,b+1} + M_{a+1,b+1|a+1,b+1} \\ 
         + M_{a,b|b,a} + M_{a+1,b|b,a+1} + M_{a,b+1|b+1,a} + M_{a+1,b+1|b+1,a+1}.
        \end{aligned}
        \end{equation*} 
    \item If $b = a - 1$, then
        \begin{equation*}
        \begin{aligned}            
        P_{a-1,a|a,a-1}= 
          M_{a-1,a|a,a-1} + M_{a,a+1|a,a+1} + M_{a,a+1|a+1,a} + M_{a-1,a+1|a+1,a-1} \\
        + M_{a-1,a+1|a-1,a+1} + M_{a-1,a|a-1,a} + M_{a,a-1|a,a-1} + M_{a+1,a|a,a+1} \\
        + M_{a+1,a|a+1,a} + M_{a+1,a-1|a-1,a+1} + M_{a+1,a-1|a+1,a-1} + 2M_{a,a|a,a} \\
        + M_{a,a-1|a-1,a}.
        \end{aligned}
        \end{equation*}       
\end{enumerate}
\end{enumerate}
\end{thm}
\subsection{Proof}
In this subsection, we prove Theorem \ref{gl221} and Theorem \ref{gl222}. We use the method of translation functors by effecting certain representations. Here are the weights of these representations.
\\
Dual $V^*$:
        \begin{equation*}
        \begin{aligned}            
        \{-\delta_1, \ -\delta_2, \ -\epsilon_1, \ -\epsilon_2\}
        \end{aligned}
        \end{equation*}
Wedge-squared of the natural $\extp^2 V$:
        \begin{equation*}
        \begin{aligned}            
        \{\delta_1 + \delta_2, \ \delta_1 + \epsilon_1, \ \delta_1 + \epsilon_2, \ \delta_2 + \epsilon_1, \\ \delta_2 + \epsilon_2, \ 2\epsilon_1, \ \epsilon_1 + \epsilon_2,  \ 2\epsilon_2\}
        \end{aligned}
        \end{equation*}
Wedge-squared of the dual $\extp^2 V^*$:
        \begin{equation*}
        \begin{aligned}            
        \{-\delta_1 - \delta_2, \ -\delta_1 - \epsilon_1, \ -\delta_1 - \epsilon_2, \ -\delta_2 - \epsilon_1, \\ -\delta_2 - \epsilon_2, \ -2\epsilon_1, \ -\epsilon_1 - \epsilon_2,  \ -2\epsilon_2\}
        \end{aligned}
        \end{equation*}
Wedge-cubed of the natural $\extp^3 V$:
        \begin{equation*}
        \begin{aligned}            
        \{\delta_1 + \delta_2 + \epsilon_1, \ \delta_1 + \delta_2 + \epsilon_2, \ \delta_1 + 2\epsilon_1, \\ 
        \delta_1 + \epsilon_1 + \epsilon_2, \ \delta_1 + 2\epsilon_2, \ \delta_2 + 2\epsilon_1, \ \delta_2 + \epsilon_1 + \epsilon_2, \\
        \delta_2 + 2\epsilon_2, \ 3\epsilon_1, \ 2\epsilon_1 + \epsilon_2, \ \epsilon_1 + 2\epsilon_2, \ 3\epsilon_2\}
        \end{aligned}
        \end{equation*}
We now offer justification for the formulae above, separated into cases that have different formulae, based on the strategy in \cref{strategy}. Our proof be more explicit in the earlier cases and edge cases; those which lack much explanation follow the strategy almost directly.

\subsubsection{Case: $P_\lambda$,  $\lambda \cong (a,b\ |\ b,a)$}
\begin{enumerate}
    \item $b < a - 1$. \\
    Let $\mu  \coloneqq \lambda - (\epsilon_1 + \epsilon_2)$, so that $\mu \cong (a, b \ | \ b+1, a+1)$. Note that $\epsilon_1 + \epsilon_2$ is not the lowest weight of $\extp^2 V$ as prescribed by the strategy; this shall be justified below. First we claim: 
        \begin{equation*}
        \begin{aligned}       
        P_{\mu} = M_{a,b|b+1,a+1}, \\     
        P \coloneqq \mathrm{Pr}_\lambda(P_{\mu}\otimes \extp^2 V) = 
          M_{a,b|b,a} + M_{a+1,b|b,a+1} + M_{a,b+1|b+1,a} \\ + M_{a+1,b+1|b+1,a+1}.
        \end{aligned}
        \end{equation*}
\hypertarget{$\star$}{Observe that $\mu$ is dominant and typical, so by Lemma \ref{typAndDom}, $P_{\mu}  = M_{a,b|b+1,a+1}$. Ttensoring $P_{\mu}$ with the wedge-squared of the natural representation produces another projective object $T$ in $\OO$. By choice of $\mu$ and representation, $M_{a,b|b,a}$ appears in the standard filtration of $T$ by Proposition \ref{sum}}.
\par
Because each indecomposable object lives entirely in a single block in $\OO$, we can project $T$ onto $\Bl_\lambda$ in $\OO$ to produce another projective $P$. Now, because $\lambda$ is the lowest weight appearing in the Verma flag of $P$, $P_{\lambda}$ is a direct summand of $P$.  
\\\\     
We now deduce which other Verma modules must appear in the filtration by the conditions in Proposition \ref{filprop}. The module $M_{a+1,b|b,a+1}$ and $M_{a,b+1|b+1,a}$ appear by condition \ref{3}. If the remaining module $M_{a+1,b+1|b+1,a+1}$ were to split off as a separate projective summand, Corollary \ref{corlen} shows us such a projective would have a filtration length of at least two. Since there are no other modules remaining, we deduce that this last module must also be present in the filtration and that there is only one summand present in the projection.

    \item $b = a - 1$. \\ There appears no obvious choice of typical $\mu$ such that tensoring one of the representations above is effective. As a result, we allow $\mu$ to be atypical. This requires us to use a suitable translation functor to first determine $P_{\mu}$. \\\\
    Let $\mu \coloneqq \lambda - 2\epsilon_2$, so that $\mu \cong (a,a-1 \ |\ a-1,a+2)$ and $\mu - \rho$ is atypical of degree $1$. Let  $\theta \coloneqq \mu - 2\epsilon_1$ so that $\theta \cong (a, a-1 \ |\ a+1, a+2)$ and $\theta - \rho$ is typical. We have
        \begin{equation*}
        \begin{aligned}            
        P_{\theta} = M_{a,a-1|a+1,a+2}, \\
        \mathrm{Pr}_\mu (P_{\theta} \otimes \extp^2 V) = 
          M_{a,a-1|a-1,a+2} + M_{a,a|a,a+2} + M_{a+1,a|a+1,a+2}.
        \end{aligned}
        \end{equation*}     
  To determine $P_{\mu}$, we first note that because $M_{\mu}$ present above has weight lower than that of the remaining Verma modules, $P_{\mu}$ is a direct summand of $P$. Now, we effect Proposition \ref{filprop} to see which of the remaining Verma modules appear in a Verma flag of $P_{\mu + \rho}$. The module $M_{a,a|a,a+2}$ appears by condition \ref{3}, and $M_{a+1,a|a+1,a+2}$ appears by condition \ref{5}. We deduce that there is only one summand in the projection.
        \begin{equation*}
        \begin{aligned}            
        P_{\mu} =  M_{a,a-1|a-1,a+2} + M_{a,a|a,a+2} + M_{a+1,a|a+1,a+2}.
        \end{aligned}
        \end{equation*}     
    Now, we argue a formula for $P_{\lambda}$ by proceeding as usual. We have
        \begin{equation*}
        \begin{aligned}            
        \mathrm{Pr}_\lambda (P_{\mu} \otimes \extp^2 V) = 
          M_{a,a-1|a-1,a} + M_{a+1,a-1|a-1,a+1} + M_{a,a|a,a} \\
        + M_{a,a+1|a,a+1} + 2M_{a+1,a|a,a+1} + M_{a+1,a|a+1,a} \\
        + M_{a+1,a+1|a+1,a+1} + M_{a+2,a|a,a+2} + M_{a+2,a+1|a+1,a+2}.
        \end{aligned}
        \end{equation*}    
    The presence of $M_\lambda = M_{a,a-1|a-1}$ in the projection and because it is of lowest weight tells us that $P_\lambda$ is a direct summand. Now, we determine which of the remaining Verma modules must appear in a filtration for $P_\lambda$ by using Proposition \ref{filprop}. $ M_{a+1,a-1|a-1,a+1}$ and $M_{a,a|a,a}$ appear by condition \ref{3}. $M_{a,a+1|a,a+1}$, one copy of $M_{a+1,a|a,a+1}$, and $M_{a+1,a|a+1,a}$ appear by condition \ref{5}. Because we have enumerated at least $6$ terms here, $P_\lambda$ has a filtration length of at least $6$.
    \\\\
    Of the remaining four terms, observe that the module of lowest weight appearing is the second copy of $M_{a+1,a|a,a+1}$. This is actually of the same class $(a,a-1\ |\ a-1,a)$, and so if it yielded another projective as a direct summand, there would be at least $12 = 6 + 6$ terms appearing. Since there are only ten terms, the second copy of $M_{a+1,a|a,a+1}$ also appears.
    \\\\
    There are three terms remaining: $M_{a+1,a+1|a+1,a+1}$, $M_{a+2,a|a,a+2}$, and $M_{a+2,a+1|a+1,a+2}$. The weights of the first two are incomparable, and the third is higher than both. Therefore, these three modules cannot form a projective. But three modules cannot also be a direct sum of two or more projectives of atypical weight by a simple application of Corollary \ref{corlen}. We deduce that all three of these modules are also present in the filtration and that there is only one summand in the projection.
    
    \item $b = a$. \\ 
    Let $\mu \coloneqq \lambda - (\epsilon_1 + \epsilon_2)$. We have
        \begin{equation*}
        \begin{aligned} 
        P_\mu = M_{a,a|a+1,a+1}, \\           
        \mathrm{Pr}_\lambda (P_\mu \otimes \extp^2 V) = 
          M_{a,a|a,a} + M_{a+1,a|a,a+1} + M_{a,a+1|a+1,a} \\
        + M_{a+1,a+1|a+1,a+1} + M_{a+1,a|a+1,a} + M_{a,a+1|a,a+1}.
        \end{aligned}
        \end{equation*} 
    The presence of $M_{a,a|a,a}$ as the module with lowest weight implies that $P_\lambda$ is a direct summand of the projection. To determine which other modules appear in a Verma flag for $P_\lambda$, we use Proposition \ref{filprop}. The modules $M_{a+1,a|a,a+1}$, $M_{a,a+1|a+1,a}$, $M_{a+1,a|a+1,a}$, and $M_{a,a+1|a,a+1}$ all appear by condition \ref{3} of the proposition. This forces the remaining module $M_{a+1,a+1|a+1,a+1}$ to also appear by Corollary \ref{corlen}. We deduce that there is only one direct summand.    
\end{enumerate}

\subsubsection{Case $P_\lambda, \ \lambda(a,b\ |\ a,b)$}
\begin{enumerate}
    \item $b < a - 1$. \\
    Let $\mu \coloneqq \lambda - (\epsilon_1 + \epsilon_2)$. We have
        \begin{equation*}
        \begin{aligned}   
        P_\mu = M_{a,b|b+1,a+1} + M_{a,b|a+1,b+1}, \\         
        \mathrm{Pr}_\lambda (P_\mu \otimes \extp^2 V) = 
          M_{a,b|a,b} + M_{a+1,b|a+1,b} + M_{a,b+1|a,b+1} + M_{a+1,b+1|a+1,b+1} \\
        + M_{a,b|b,a} + M_{a+1,b|b,a+1} + M_{a,b+1|b+1,a} + M_{a+1,b+1|b+1,a+1}.
        \end{aligned}
        \end{equation*}
        The presence of $M_{a,b|a,b}$, which has lowest weight in the projection, implies $P_\lambda$ is a direct summand. As before, we now apply Proposition \ref{filprop} to see which modules appear in the standard filtration. Modules $M_{a+1,b|a+1,b}$ and $M_{a,b+1|a,b+1}$ appear by condition \ref{3}. The module $M_{a,b|b,a}$ appears by condition \ref{1}. The modules $M_{a+1,b|b,a+1}$ and $M_{a,b+1|b+1,a}$ appear by condition \ref{4}. We have $6$ terms so far. The modules $M_{a+1,b+1|a+1,b+1}$ and $M_{a+1,b+1|b+1,a+1}$, are the two remaining terms, with the former of lower weight.  But observe that its weight is in the same class $(a,b|a,b)$ with $b < a-1$. So if we were to have another projective, we would have at least $12 = 6 + 6$ terms. There are only $8$ terms, so the module $M_{a+1,b+1|a+1,b+1}$ must also appear. It follows $M_{a+1,b+1|b+1,a+1}$ must appear by Corollary \ref{corlen}. We deduce that there is only one direct summand.
        
    \item $b = a - 1$. \\ 
    Let $\mu \coloneqq \lambda - (\epsilon_1 + 2\epsilon_2)$. We have
        \begin{equation*}
        \begin{aligned}           
        P_\mu = M_{a,a-1|a+1,a+1}, \\ 
        \mathrm{Pr}_\lambda (P_\mu \otimes \extp^3 V) = 
          M_{a,a-1|a,a-1} + M_{a+1,a|a,a+1} + M_{a+1,a|a+1,a} + M_{a+1,a-1|a-1,a+1} \\
        + M_{a+1,a-1|a+1,a-1} + M_{a,a|a,a} + M_{a,a-1|a-1,a}.
        \end{aligned}
        \end{equation*} 
        The presence of $M_{a,a-1|a,a-1}$, which has lowest weight in the projection, implies $P_\lambda$ is a direct summand. As before, we apply Proposition \ref{filprop} to see which modules appear in the standard filtration. The module $M_{a+1,a|a,a+1}$ must appear by condition \ref{5}. The module $M_{a+1,a-1|a-1,a+1}$ appears by condition \ref{4}. The modules $M_{a+1,a-1|a+1,a-1}$ and $M_{a,a|a,a}$ appear by condition \ref{3}. $M_{a,a-1|a-1,a}$ appears by condition \ref{1}. Lastly, $M_{a+1,a|a+1,a}$ must also appear in the standard filtration by Corollary \ref{corlen}. We deduce that there is only one direct summand.

    \item $b = a$. This is a repeat of an earlier case.
\end{enumerate}

\subsubsection{Case $P_\lambda, \ \lambda \cong (b,a\ |\ b,a)$}
\begin{enumerate}
    \item $b < a - 1$.
    Let $\mu \coloneqq \lambda - (\epsilon_1 + \epsilon_2)$. We have
        \begin{equation*}
        \begin{aligned}
        P_\mu = M_{b,a|b+1,a+1} + M_{a,b|b+1,a+1},  \\         
        \mathrm{Pr}_\lambda (P_\mu \otimes \extp^2 V) = 
           M_{b,a|b,a} + M_{b,a+1|b,a+1} + M_{b+1,a|b+1,a} + M_{b+1,a+1|b+1,a+1} \\
         + M_{a,b|b,a} + M_{a+1,b|b,a+1} + M_{a,b+1|b+1,a} + M_{a+1,b+1|b+1,a+1}.
        \end{aligned}
        \end{equation*}         
        The presence of $M_{b,a|b,a}$, which has lowest weight in the projection, implies $P_\lambda$ is a direct summand. We now apply Proposition \ref{filprop} to deduce which of the other modules appear in the filtration. The modules $M_{b,a+1|b,a+1}$ and $M_{b+1,a|b+1,a}$ appear by condition \ref{3}. The module $M_{a,b|b,a}$ appears by condition \ref{2}. The modules $M_{a+1,b|b,a+1}$ and $M_{a,b+1|b+1,a}$ appear by condition \ref{4}. We have $6$ terms so far. The modules $M_{b+1,a+1|b+1,a+1}$ and $M_{a+1,b+1|b+1,a+1}$, are the two remaining terms, with the former of lower weight. Observe that this weight is of the same class $(b,a|b,a)$ with $b < a-1$. So if this were to yield another projective, we would have at least $12 = 6 + 6$ terms. So $M_{b+1,a+1|b+1,a+1}$ must appear. It follows $M_{a+1,b+1|b+1,a+1}$ must appear by Corollary \ref{corlen}. We deduce that there is only one direct summand.
                
    \item $b = a - 1$. \\ 
    There appears no obvious choice of typical $\mu$ such that tensoring one of the representations above is effective. As a result, we allow $\mu$ to be atypical. This requires us to use a suitable translation functor to first determine $P_\mu$. \\\\
    Let $\mu \coloneqq \lambda - (-\delta_1)$, so that $\mu \cong (a, a | a-1, a)$ and $\mu - \rho$ is atypical of degree $1$. Let $\theta \coloneqq \mu - (-\delta_1 - \delta_2)$, so $\theta \cong (a+1, a+1 |a-1, a)$ and $\theta - \rho$ is typical. We have the following:
        \begin{equation*}
        \begin{aligned}            
        P_\theta = M_{a+1,a+1 | a-1, a}, \\
        \mathrm{Pr}_\mu (P_\theta \otimes  \extp^2 V^*) = 
          M_{a,a |a-1, a} + M_{a+1, a | a-1, a+1} + M_{a, a+1 |a-1, a+1}.
        \end{aligned}
        \end{equation*}    
        The presence of $M_{a,a |a-1, a}$, which has lowest weight in the projection, indicates that $P_\mu$ is a direct summand. Applying Proposition \ref{filprop} to see which modules are present in the filtration for $P_\mu$, we see that. Modules $M_{a+1, a | a-1, a+1}$ and $M_{a, a+1 |a-1, a+1}$ appear by condition \ref{3}. We deduce that there is only one summand and that
        \begin{equation*}
        \begin{aligned}            
         P_\mu =  M_{a,a |a-1, a} + M_{a+1, a | a-1, a+1} + M_{a, a+1 |a-1, a+1}.
        \end{aligned}
        \end{equation*}           
        Now, we argue a formula for $P_\lambda$. We have
        \begin{equation*}
        \begin{aligned}            
        \mathrm{Pr}_\lambda (P_\mu \otimes  V^*) = 
          M_{a-1,a|a-1,a} + M_{a,a-1|a-1,a} + M_{a,a|a,a} \\
        + M_{a-1,a+1|a-1,a+1} + M_{a,a+1|a,a+1} + M_{a+1,a-1|a-1,a+1} \\
        + M_{a+1,a|a,a+1}.
        \end{aligned}
        \end{equation*}
    The presence of $M_{a-1,a |a-1, a}$, which has lowest weight in the projection, indicates that $P_\lambda$ is a direct summand. We proceed as before with Proposition \ref{filprop} to see which modules appear in the Verma flag. Module $M_{a,a-1|a-1,a}$ appears by condition \ref{1}. Modules $M_{a,a|a,a}$ and $M_{a-1,a+1|a-1,a+1}$ appear by condition \ref{3}. $M_{a+1,a-1|a-1,a+1}$ appears by condition \ref{4}, and $M_{a+1,a|a,a+1}$ appears by condition \ref{5}.  It follows $M_{a,a+1|a,a+1}$ must appear by Corollary \ref{corlen}. We deduce that there is only one summand.
    
    \item $b = a$. This is a repeat of an earlier case.
   
\end{enumerate}

\subsubsection{Case $P_\lambda, \ \lambda \cong (b,a\ |\ a,b)$}
\begin{enumerate}
    \item $b < a - 1$. \\ 
    Let $\mu \coloneqq \lambda - (\delta_1 + \delta_2)$. We have
        \begin{equation*}
        \begin{aligned}            
         P_\mu = M_{b,a|a+1,b+1} + M_{b,a|b+1,a+1} + M_{a,b|a+1,b+1} + M_{a,b|b+1,a+1}, \\        
        \mathrm{Pr}_\lambda (P_\mu \otimes \extp^2 V) = 
           M_{b,a|a,b} + M_{b,a+1|a+1,b} + M_{b+1,a|a,b+1} + M_{b+1,a+1|a+1,b+1} \\
         + M_{b,a|b,a} + M_{b,a+1|b,a+1} + M_{b+1,a|b+1,a} + M_{b+1,a+1|b+1,a+1} \\   
         + M_{a,b|a,b} + M_{a+1,b|a+1,b} + M_{a,b+1|a,b+1} + M_{a+1,b+1|a+1,b+1} \\ 
         + M_{a,b|b,a} + M_{a+1,b|b,a+1} + M_{a,b+1|b+1,a} + M_{a+1,b+1|b+1,a+1}.
        \end{aligned}
        \end{equation*}    
      The presence of $M_{b,a |a, b}$, which has lowest weight in the projection, indicates that $P_\lambda$ is a direct summand. We proceed as before with Proposition \ref{filprop} to see which modules appear in the Verma flag. Modules $M_{b,a|b,a}$ and $M_{a,b|a,b}$ appear by condition \ref{1}. $M_{a,b|b,a}$ appears by condition \ref{2}. The module $M_{b,a+1|a+1,b}$ appears by condition \ref{3}, and modules $M_{b,a+1|b,a+1}$, $M_{a+1,b|a+1,b}$, and $M_{a+1,b|b,a+1}$ appear by condition \ref{4}. The module $M_{b+1,a|a,b+1}$ appears by condition \ref{3}, and $M_{b+1,a|b+1,a}$, $M_{a,b+1|a,b+1}$, and $M_{a,b+1|b+1,a}$ appear by condition \ref{4}. $M_{b+1,a+1|a+1,b+1}$ appears by condition \ref{5}, and $M_{b+1,a+1|b+1,a+1}$, $M_{a+1,b+1|a+1,b+1}$, and $M_{a+1,b+1|b+1,a+1}$ appear by condition \ref{6}. We deduce that there is only one summand.
       
    \item $b = a - 1$.
    Let $\mu \coloneqq \lambda - (\epsilon_1 + 2\epsilon_2)$. We have
        \begin{equation*}
        \begin{aligned}            
        P_\mu = M_{a-1,a|a+1,a+1} + M_{a,a-1|a+1,a+1}, \\
        \mathrm{Pr}_\lambda (P_\mu \otimes \extp^3 V) = 
          M_{a-1,a|a,a-1} + M_{a,a+1|a,a+1} + M_{a,a+1|a+1,a} + M_{a-1,a+1|a+1,a-1} \\
        + M_{a-1,a+1|a-1,a+1} + M_{a-1,a|a-1,a} + M_{a,a-1|a,a-1} + M_{a+1,a|a,a+1} \\
        + M_{a+1,a|a+1,a} + M_{a+1,a-1|a-1,a+1} + M_{a+1,a-1|a+1,a-1} + 2M_{a,a|a,a} \\
        + M_{a,a-1|a-1,a}.
        \end{aligned}
        \end{equation*}
	The presence of $M_{a-1,a |a, a-1}$, which has lowest weight in the projection, indicates that $P_\lambda$ is a direct summand. We proceed as before with Proposition \ref{filprop} to see which modules appear in the Verma flag. The module $M_{a-1,a+1|a+1,a-1}$ and at least one copy of $M_{a,a|a,a}$ appear by condition \ref{3}. The module $M_{a,a+1|a+1,a}$ appears by condition \ref{5}. The modules $M_{a-1,a|a-1,a}$ and $M_{a,a-1|a,a-1}$ appear by condition \ref{1}. The module $M_{a,a-1|a-1,a}$ appears by condition \ref{2}. The modules $M_{a-1,a+1|a-1,a+1}$, $M_{a+1,a-1|a-1,a+1}$, and $M_{a+1,a-1|a+1,a-1}$ appear by condition \ref{4}. The modules $M_{a,a+1|a,a+1}$, $M_{a+1,a|a,a+1}$, and $M_{a+1,a|a+1,a}$ appear by condition \ref{6}. This forces the second copy of $M_{a,a|a,a}$ to appear by Corollary \ref{corlen}. We deduce that there is only one summand.
                    
    \item $b = a$. This is a repeat of an earlier case.
\end{enumerate}

\section{Jordan-H{\"o}lder Formulae for $\gl(2|2)$}\label{sec7}
By BGG reciprocity, we can convert the standard filtration multiplicities for projective modules into Jordan-H{\"o}lder multiplicities of irreducible modules for Verma modules. These results are presented in the two theorems below.
\begin{thm}\label{gl221c}
Let $a,b \in \mathbb{Z}$ with $a \geq b$. We have the following irreducible decomposition formulae for the Verma modules $M_{a, b|b, a}$ and $M_{a, b|a, b}$ in $\Bl_0$:
\begin{enumerate}[label=(\arabic*), ref=(\arabic*]
\item Case: $M_{a,b|b,a}$
\begin{enumerate}[label=\theenumi.\arabic*), ref = \arabic*]
    \item If $b < a - 2$, then
        \begin{equation*}
        \begin{aligned}            
        M_{a,b|b,a}  = 
           L_{a,b|b,a} + L_{a,b|a,b} + L_{b,a|b,a} + L_{b,a|a,b} \\
         + L_{a-1,b|b,a-1} + L_{a-1,b|a-1,b} + L_{b,a-1|b,a-1} + L_{b,a-1|a-1,b} \\
         + L_{a,b-1|b-1,a} + L_{a,b-1|a,b-1} + L_{b-1,a|b-1,a} + L_{b-1,a|a,b-1} \\
         + L_{a-1,b-1|b-1,a-1} + L_{a-1,b-1|a-1,b-1} + L_{b-1,a-1|b-1,a-1} + L_{b-1,a-1|a-1,b-1}.
        \end{aligned}
        \end{equation*}
    \item If $b = a - 2$, then
        \begin{equation*}
        \begin{aligned}            
        M_{a,a-2|a-2,a}  = 
           L_{a,a-2|a-2,a} + L_{a,a-2|a,a-2} + L_{a-2,a|a-2,a} + L_{a-2,a|a,a-2} \\
         + L_{a,a-3|a-3,a} + L_{a,a-3|a,a-3} + L_{a-3,a|a-3,a} + L_{a-3,a|a,a-3} \\
         + L_{a-1,a-3|a-3,a-1} + L_{a-1,a-3|a-1,a-3} + L_{a-3,a-1|a-3,a-1} + L_{a-3,a-1|a-1,a-3} \\
         + L_{a-1,a-2|a-2,a-1} + L_{a-1,a-2|a-1,a-2} + L_{a-2,a-1|a-2,a-1} + L_{a-2,a-1|a-1,a-2} \\
         + L_{a-2,a-3|a-3,a-2}.
        \end{aligned}
        \end{equation*}    
    \item If $b = a - 1$, then
        \begin{equation*}
        \begin{aligned}            
        M_{a,a-1|a-1,a}= 
           L_{a,a-1|a-1,a} + L_{a-1,a|a-1,a} + L_{a,a-1|a,a-1} + L_{a-1,a|a,a-1} \\
         + 2L_{a-1,a-2|a-2,a-1} + L_{a-2,a-1|a-2,a-1} + L_{a-1,a-2|a-1,a-2} + L_{a-2,a-1|a-1,a-2} \\
         + L_{a-1,a-1|a-1,a-1} + L_{a-2,a-3|a-3,a-2}.
        \end{aligned}
        \end{equation*}        
    \item If $b = a$, then
        \begin{equation*}
        \begin{aligned}            
        M_{a,a|a,a}= 
           L_{a,a-1|a-1,a} + L_{a,a|a,a} + L_{a,a-1|a,a-1} + L_{a-1,a|a-1,a} \\
         + 2L_{a-1,a|a,a-1} + L_{a-1,a-2|a-2,a-1} + L_{a-1,a-1|a-1,a-1}
        \end{aligned}
        \end{equation*}   
\end{enumerate}

\item Case: $M_{a,b|a,b}$
\begin{enumerate}[label=\theenumi.\arabic*), ref = \arabic*]
    \item If $b < a - 2$, then
        \begin{equation*}
        \begin{aligned}            
        M_{a,b|a,b}  = 
           L_{a,b|a,b} + L_{b,a|a,b} + L_{a-1,b|a-1,b} + L_{b,a-1|a-1,b} \\
         + L_{a,b-1|a,b-1} + L_{b-1,a|a,b-1} + L_{a-1,b-1|a-1,b-1} + L_{b-1,a-1|a-1,b-1}.
        \end{aligned}
        \end{equation*}
    \item If $b = a - 2$, then
        \begin{equation*}
        \begin{aligned}            
        M_{a,a-2|a,a-2}  = 
           L_{a,a-2|a,a-2} + L_{a-2,a|a,a-2} + L_{a-1,a-2|a-1,a-2} + L_{a-2,a-1|a-1,a-2} \\
         + L_{a,a-3|a,a-3} + L_{a-3,a|a,a-3} + L_{a-1,a-3|a-1,a-3} + L_{a-3,a-1|a-1,a-3}.
        \end{aligned}
        \end{equation*}    
    \item If $b = a - 1$, then
        \begin{equation*}
        \begin{aligned}            
        M_{a,a-1|a,a-1}= 
           L_{a,a-1|a,a-1} + L_{a-1,a|a,a-1} + L_{a-1,a-2|a-2,a-1} + L_{a-1,a-1|a-1,a-1} \\
         + L_{a-1,a-2|a-1,a-2} + L_{a-2,a-1|a-1,a-2}.
        \end{aligned}
        \end{equation*}        
\end{enumerate}
\end{enumerate}
\end{thm}

\begin{thm}\label{gl221d}
Let $a,b \in \mathbb{Z}$ with $a \geq b$. We have the following irreducible decomposition formulae for the Verma modules $M_{b, a|b, a}$ and $M_{b, a|a, b}$ in $\Bl_0$:
\begin{enumerate}[label=(\arabic*), ref=(\arabic*]
\item Case: $M_{b,a|b,a}$
\begin{enumerate}[label=\theenumi.\arabic*), ref = \arabic*]
    \item If $b < a - 2$, then
        \begin{equation*}
        \begin{aligned}            
        M_{b,a|b,a}  = 
           L_{b,a|b,a} + L_{b,a|a,b} + L_{b,a-1|b,a-1} + L_{b,a-1|a-1,b} \\
         + L_{b-1,a|b-1,a} + L_{b-1,a|a,b-1} + L_{a-1,b-1|a-1,b-1} + L_{b-1,a-1|a-1,b-1}.
        \end{aligned}
        \end{equation*}
    \item If $b = a - 2$, then
        \begin{equation*}
        \begin{aligned}            
        M_{a-2,a|a-2,a}  = 
           L_{a-2,a|a-2,a} + L_{a-2,a|a,a-2} + L_{a-2,a-1|a-2,a-1} + L_{a-2,a-1|a-1,a-2} \\
         + L_{a-3,a|a-3,a} + L_{a-3,a|a,a-3} + L_{a-3,a-1|a-3,a-1} + L_{a-3,a-1|a-1,a-3}.
        \end{aligned}
        \end{equation*}    
    \item If $b = a - 1$, then
        \begin{equation*}
        \begin{aligned}            
        M_{a,a-1|a,a-1}= 
           L_{a-1,a|a-1,a} + L_{a-1,a|a,a-1} + L_{a-1,a-2|a-2,a-1} + L_{a-1,a-1|a-1,a-1} \\
         + L_{a-2,a-1|a-2,a-1} + L_{a-2,a-1|a-1,a-2}.
        \end{aligned}
        \end{equation*}        
\end{enumerate}

\item Case: $M_{b,a|a,b}$
\begin{enumerate}[label=\theenumi.\arabic*), ref = \arabic*]
    \item If $b < a - 2$, then
        \begin{equation*}
        \begin{aligned}            
        M_{b,a|a,b}  = 
           L_{b,a|a,b} + L_{b,a-1|a-1,b} + L_{b-1,a|a,b-1} + L_{b-1,a-1|a-1,b-1}.
        \end{aligned}
        \end{equation*}
    \item If $b = a - 2$, then
        \begin{equation*}
        \begin{aligned}            
        M_{a-2,a|a,a-2}  = 
           L_{a-2,a|a,a-2} + L_{a-3,a|a,a-3} + L_{a-3,a-1|a-1,a-3} + L_{a-2,a-1|a-1,a-2}.
        \end{aligned}
        \end{equation*}    
    \item If $b = a - 1$, then
        \begin{equation*}
        \begin{aligned}            
        M_{a-1,a|a,a-1}= 
           L_{a-1,a|a,a-1} + L_{a-1,a-1|a-1,a-1} + L_{a-2,a-1|a-1,a-2}.
        \end{aligned}
        \end{equation*}        
\end{enumerate}
\end{enumerate}
\end{thm}

%%%%%%%%%%%%%%

\end{document}